\documentclass[12pt]{amsart}
\usepackage{amsfonts,amssymb,amscd,amsmath,enumerate,verbatim,calc}
\usepackage{hyperref}
\usepackage[all]{xy}
\usepackage[mathscr]{euscript} 
\usepackage{ mathrsfs }
\usepackage[left=2.25cm,right=2.25cm,top=2.5cm,bottom=2.5cm]{geometry}

\usepackage{graphicx}
\usepackage{color}

\newcommand{\Tfae}{The following conditions are equivalent:}

\newcommand{\m}{\mathfrak{m} }

\newcommand{\p}{\mathfrak{p}}
\newcommand{\FF}{\mathcal{F}}

\providecommand{\e}{{\mathcal E}}

\newcommand{\Z}{\mathbb{Z} }
\newcommand{\N}{\mathbb{N} }
\newcommand{\TT}{\mathcal{T}}

\newcommand{\rt}{\rightarrow}
\newcommand{\lrt}{\longrightarrow}

\newcommand{\im}{\operatorname{image}}

\newcommand{\Ass}{\operatorname{Ass}}
\newcommand{\chars}{\operatorname{char}}
\newcommand{\depth}{\operatorname{depth}}
\newcommand{\grade}{\operatorname{grade}}

\providecommand\Mod{\text{\rm Mod}}

\newcommand{\Supp}{\operatorname{Supp}}
\providecommand\Spec{\text{\rm Spec}}

\newcommand{\injdim}{\operatorname{injdim}}

\newcommand{\cd}{\operatorname{cd}}

\theoremstyle{plain}

\newtheorem{theorem}{Theorem}[section]
\newtheorem{corollary}[theorem]{Corollary}
\newtheorem{lemma}[theorem]{Lemma}
\newtheorem{proposition}[theorem]{Proposition}

\theoremstyle{definition}
\newtheorem{definition}[theorem]{Definition}
\newtheorem{s}[theorem]{}
\newtheorem{remark}[theorem]{Remark}

\newtheorem*{example*}{\it Example}

\theoremstyle{remark}
\newtheorem{claim}{\bf Claim}
\newtheorem*{claim*}{\it Claim}

\newtheorem*{case*}{\it Case}
\newtheorem*{note*}{\it Note}

\title[Graded components of local cohomology modules]{Graded components of local cohomology modules II}
\address{Department of Mathematics, Indian Institute of Technology Bombay, Powai, Mumbai 400 076, India}
\author{Tony J. Puthenpurakal}
\email{\href{mailto:tputhen@math.iitb.ac.in}{tputhen@math.iitb.ac.in}}
\author{Sudeshna Roy}
\email{\href{mailto:sudeshnaroy.11@gmail.com}{sudeshnaroy.11@gmail.com}}
\date{\today}

\begin{document}
\begin{abstract}
Let $A$ be a commutative Noetherian ring containing a field 
of characteristic zero.
Let $R= A[X_1, \ldots, X_m]$ be a polynomial ring and 
$A_m(A) = A \langle X_1, \ldots, X_m, \partial_1, \ldots, \partial_m \rangle$ be the $m^{th}$ Weyl algebra over $A$, where $\partial_i=\partial/\partial X_i$. Consider both $R$ and $A_m(A)$ as standard graded with $\deg z=0$ for all $z \in A$, $\deg X_i=1$, and $\deg \partial_i =-1$ 
for $i=1, \ldots, m$. We present a few results about the behavior of the graded components of local cohomology modules $H_I^i(R)$, where $I$ is an arbitrary homogeneous ideal in $R$. We mostly restrict our attention to the Vanishing, Tameness, and Rigidity 
properties. To obtain this, we use the theory of $D$-modules and show that generalized Eulerian $A_m(A)$-modules exhibit these properties. As 
a corollary, we further get that components of graded local cohomology modules with respect to a pair of ideals display similar behavior.
\end{abstract}

\maketitle
\section{introduction}	
Let $L$ be a finitely generated graded module over a standard graded ring $T = \bigoplus_{n \geq 0} T_n$ with irrelevant ideal $T_+ = \bigoplus_{n > 0} T_n$. Normally, the graded components of the local cohomology module $H^i_{T_+}(L)$ are 
well-behaved. They are finitely generated
modules over the base ring $T_0$ and vanish for large $n$.
It is therefore natural to expect that
similar behavior carried out by the local cohomology modules supported on arbitrary
homogeneous ideals. On the contrary, 
this is not true in general, see \cite[Exercise 15.1.8]{BS}. In Part I of this paper \cite{TP2}, the first author 
studied the components of graded local cohomology modules $H^i_I(R)$ in detail, where $I$ is an arbitrary homogeneous ideal of a standard graded polynomial ring $R = R_0[X_1,\ldots, X_m]$ over a regular ring $R_0$ containing a field of characteristic zero. He observed that they are notably well-behaved. In this 
part, we 
extend some of his results to a relatively wider framework, where $R_0$ is any commutative Noetherian ring containing a field of characteristic zero. We also give examples to show that Theorems 1.7, 1.8, 1.9, 1.13, 1.14 
in \cite{TP2} are false if we don't assume that $R_0$ is regular.

We first summarize our findings under the following setup.

\s\label{sa}{\it Standard assumption:} 
Throughout this paper $A$ is a commutative Noetherian ring containing a field $K$ of characteristic zero. Let $R= A[X_1, \ldots, X_m]$ be standard graded with $\deg z=0$ for all $z \in A$ and $\deg X_i=1$ for $i=1, \ldots, m$. We also assume $m \geq 1$. Let $I$ be a homogeneous ideal in $R$. 
It is well-known that $H_I^i(R)$ is a graded $R$-module. We set $M:=H_I^i(R)=\bigoplus_{n\in \Z}M_n$.

\vspace{0.2cm}
\noindent
{\bf I.} ({\it Vanishing}): Let $B$ be a commutative Noetherian ring, $J$ be an ideal of $B$, and $N$ be a $B$-module. We all know that $H^j_J(N)=0$ for all $j> \dim N$ 
and $j< \grade_J N$, see \cite[Theorem 6.1.2, Theorem 6.2.7]{BS}. 
Since $H^j_J(N)$ is rarely finitely generated, usually it is difficult to tell whether 
$H^j_J(N)$ is zero if $\grade_J N < j < \dim N$. Nevertheless, two completely different algorithms have been produced 
to test whether $H^j_J(N)$ vanish or not for the case where $N=B=K[X_1, \ldots, X_m]$ is a polynomial ring over a field $K$ with $\chars K=0$ \cite{UW} and $\chars K=p>0$ \cite{Lyu-2}, respectively.

The following theorem states that vanishing of almost all graded components of $H^i_I(R)$ implies vanishing of $H^i_I(R)$.

\begin{theorem}\label{vanish}
{\rm(}with hypotheses as in \ref{sa}{\rm)}.  If $M _n = 0$ for all $|n|  \gg 0$, then $M = 0$.
\end{theorem}

Notice that if
$M = H^i_I(R) = \bigoplus_{n \in \Z}M_n $ is \emph{non-zero} then either $M_n \neq 0$ for infinitely  many $n \ll 0$, OR, $M_n \neq 0$ for infinitely  many $n \gg 0$.

\vspace{0.2cm}
\noindent
{\bf II.} ({\it Tameness}): 
In \cite{BrodHel}, M. P. Brodmann and M. Hellus raised the most challenging question related to the asymptotic behavior of local cohomology, which is the so-called {\it tameness problem}, that is, whether there exists an integer $n_0$ such that $H^i_{T_+}(L)_n$ vanishes for all $n \leq n_0$, or else does not vanish for all $n \leq n_0$. This question has been answered affirmatively in many cases, in particular, if $i = \cd(L, T_+) = \max\{j : H^j_{T_+}(L) \neq 0\}$, the cohomological dimension of $L$ with respect to $T_+$, and if $\dim T_0 \geq 2$, see \cite{BrodHel}, \cite{Lim}, \cite{RotSeg}, \cite{CJR}. S. D. Cutkosky and J. Herzog \cite{CutHer} constructed an example to show that the above question has a negative answer if $\dim T_0 = 3$.

We show that $M$ is \emph{tame}, that is, the tameness problem has a positive answer when $L = T = R$.
More precisely, 

\begin{theorem}\label{tame}
	{\rm(}with hypotheses as in \ref{sa}{\rm)}. Then we have
	\begin{enumerate}[\rm (a)]
		\item
		The following assertions are equivalent:
		\begin{enumerate}[\rm(i)]
			\item
			$M_n \neq 0$ for infinitely many $n \ll 0$.
			\item
			$M_n \neq 0$ for all $n \leq -m$.
		\end{enumerate}
		\item
		The following assertions are equivalent:
		\begin{enumerate}[\rm(i)]
			\item
			$M_n \neq 0$ for infinitely many $n \gg 0$.
			\item
			$M_n \neq 0$ for all $n \geq 0$.
		\end{enumerate}
	\end{enumerate}
\end{theorem}

\vspace{0.2cm}
\noindent
{\bf III.} ({\it Rigidity}): 
Surprisingly non-vanishing of a single graded component of $M = H^i_I(R)$ is very strong. We prove the following rigidity result:
\begin{theorem}\label{rigid-intro}
	{\rm(}with hypotheses as in \ref{sa}{\rm)}. Then we have
	\begin{enumerate}[\rm (a)]
		\item
		The following assertions are equivalent:
		\begin{enumerate}[\rm(i)]
			\item
			$M_r \neq 0$ for  some $r \leq -m$.
			\item
			$M_n \neq 0$ for all $n \leq -m$.
		\end{enumerate}
		\item
		The following assertions are equivalent:
		\begin{enumerate}[\rm(i)]
			\item
			$M_s \neq 0$ for some $s \geq 0$.
			\item
			$M_n \neq 0$ for all $n \geq 0$.
		\end{enumerate}
		\item
		{\rm(}When $m \geq 2$.{\rm)} 
		The following assertions are equivalent:
		\begin{enumerate}[\rm(i)]
			\item
			$M_t \neq 0$ for some $t$ with $-m < t < 0$.
			\item
			$M_n \neq 0$ for all $n \in \mathbb{Z}$.
		\end{enumerate}
	\end{enumerate}
\end{theorem}

\makeatletter
\edef\orig@output{\the\output}
\output{\setbox\@cclv\vbox{\unvbox\@cclv\vspace{0pt plus 20pt}}\orig@output}
\makeatother
The above result says that the local cohomology module $H^i_I(R)$ must vanish if both $H^i_I(R)_0$ and $H^i_I(R)_{-m}$ are zero.

\vspace{0.2cm}
\noindent
\emph{Techniques used to prove our results:}
The main idea in \cite{TP2} is that if $A = K[[Y_1,\ldots, Y_d]]$ then graded local cohomology modules over $R = A[X_1,\ldots, X_m]$ become holonomic over an appropriate ring of differential operators, see \cite[Theorem 4.2]{TP2}. Furthermore, local cohomology modules are ``\textit{generalized Eulerian}", see \cite[Theorem 3.6]{TP2}
(we will discuss this property later). One of the observation we first made was that in Theorems 1.2, 1.3 and 1.6 in \cite{TP2}, the generalized Eulerian property of local cohomology modules was used more crucially. 
Additionally, in \cite{TP2}, the generalized Eulerian property of local cohomology modules was proved very generally.

We now discuss the techniques used to prove our results: Let $A$ be a Noetherian ring containing a field of characteristic zero. Let $R = A[X_1,\ldots,X_m]$
and $A_m(A)$ be the $m^{th}$ Weyl algebra on $A$. 
We give standard grading on $R$ and $A_m(A)$. Notice that $R$ is a graded subring of $A_m(A)$. If $W$ is a graded $A_m(A)$-module and $w$ is a homogeneous element of $W$, then set $|w| = \deg w$. 

Consider the Eulerian operator $\e_m = \sum_{i = 1}^{m} X_i \partial_i$. If $f \in R$ is homogeneous, then it is easy to check that $\e_m f = |f| f$. We say a graded $A_m(A)$-module $W$ is \textit{Eulerain} if $\e_m w = |w| w$ for each homogeneous element $w$ of $W$. 
Clearly, $R$ is an Eulerian $A_m(A)$-module. We say $W$ is \textit{generalized Eulerian} if for each homogeneous $w$ of $W$ there exists $a$ depending on $w$ such that
$(\e_m - |w|)^aw = 0$.

The notion of Eulerian modules was introduced in the case when $A$ is a field $K$ by L. Ma and W. Zhang \cite{MZ} (they also defined the notion of Eulerian  $D$-modules in characteristic $p > 0$, where $D$ is the ring of $K$-linear differential operators on $R = K[X_1,\ldots, X_m]$). Unfortunately however the class of Eulerian $D$-modules is not closed under extensions (see 3.5(1) in \cite{MZ}). To rectify this, the first author introduced the notion of generalized Eulerian $D$-modules (in characteristic zero), see \cite{TP6}.

The next technique that we use is the technique of de Rham cohomology, Koszul homology of generalized Eulerian modules. We generalize several properties which were proved when $A = K$
and the relevant module was holonomic over $A_m(K)$. It is worth pointing out that generalized Eulerian $A_m(K)$-modules 
are not necessarily holonomic, see \cite[Remark 3.6]{MZ}.

The final technique we use to prove vanishing and tameness property of local cohomology modules is that local cohomology modules are countably generated over $R = A[X_1,\ldots, X_m]$. We can exploit this fact if $A$ contains an uncountable field which one can always assume by base change, see Remark \ref{uncountable}. This was observed first in \cite{TP4}. 

Practically, all the proofs in this paper, except the last section, use techniques that were first developed by several researchers (including the first author) for holonomic modules over some appropriate ring of differential operators. In the generality we consider, we do not have the notion of holonomic modules. However we have the notion of generalized  Eulerian modules. After carefully rewriting several proofs we were able to deduce our results.

Let $B$ be a commutative Noetherian ring, and $I, J$ be ideals in $B$. Let $N$ be a $B$-module. In \cite{TYY}, R. Takahashi et al introduced the notion of local cohomology module of $N$ with respect to the pair of ideals $(I, J)$ and denoted it by $H^i_{I,J}(N)$, see \ref{recall-Lc-pair} \ref{itm:1}. They noticed that the local cohomology functor $H^i_{I,J}(-)$ coincides with a functor defined by H. Brenner in \cite{Brenner}. Let $T$ be a graded ring, $I$ be a homogeneous ideal, and $J$ be an arbitrary ideal in $T$. Let $L$ be a graded $T$-module. In \cite{LimPer}, P. H. Lima and V. H. Jorge P\'erez 
defined a grading on $\Gamma_{I,J}(L)$, the $(I,J)$-torsion submodule of $N$, which induces a grading on $H^i_{I,J}(L)$. They mostly studied the graded components of $H^i_{T_+,J}(L)$ and extended some classical results about the components of $H^i_{T_+}(L)$ to $H^i_{T_+,J}(L)$ when 
$T$ is a standard graded Noetherian ring, $J$ is an ideal generated by elements of degree zero, and the module $L$ is finitely generated. But, they didn't talk about any graded version of the generalized \v Cech complex, presented in \cite{TYY}. 
In the last section, we do so when both $I,J$ are homogeneous ideals in $T$. 
We use it to show that if $T=R$ as in \ref{sa} and $L$ is a generalized Eulerian $A_m(A)$-module, then $H^i_{I,J}(L)$ is generalized Eulerian. Consequently, it is obtained that the vanishing, tameness, and rigidity properties hold for $H^i_{I,J}(L)$.

The paper is structured as follows. Section 2 
consists of basic preliminaries that we need. In Section 3, we talk about de Rham cohomology and Koszul cohomology of generalized Eulerian $A_m(A)$-modules. In Section 4, some properties of $A_m(A)$-modules are discussed. 
In Section 5, we show the countable generation of local cohomology modules. In Section 6 and Section 7, we prove our main results regarding vanishing, tameness, and rigidity property of generalized Eulerian $A_m(A)$-modules. In Section 8, an application of our theory to graded local cohomology modules are considered. 
In Section 9, we give examples to show that Theorems 1.7, 1.8, 1.9, 1.13, 1.14 in \cite{TP2} are false if we don't assume $A$ is regular. 
The last section is devoted for the study of 
graded local cohomology module with respect to a pair of ideals.

\section{Preliminaries}
In this section, we recall the definition of graded Lyubeznik functors and discuss its behavior under flat extension. 

\s\label{LyuDef} {\it Lyubeznik functors}

Let $B$ be a commutative Noetherian ring.
If $M$ is a $B$-module and $Y$ is a locally closed subscheme of $\Spec(B)$, we denote by $H^i_I(M)$ the $i^{th}$ local cohomology module of $M$ supported in $Y$. Suppose $Y=Y_2-Y_1$ where $Y_1 \subset Y_2$ are two closed subsets of $\Spec(B)$, then we have an exact sequence of functors 
\begin{equation}\label{lyu}
\cdots \to H_{Y_1}^i(-) \to H_{Y_2}^i(-) \to H_Y^i(-) \to H_{Y_1}^{i+1}(-) \to \cdots.
\end{equation}
A {\it Lyubeznik functor} $\TT$ is any functor of the form $\TT=\TT_1 \circ \TT_2 \circ \cdots \circ\TT_r$ where each $\TT_i$ is either $H_Y^i(-)$ for some locally closed subset $Y$ of $\Spec(B)$ or the kernel, image, or cokernel of some arrow appearing in \eqref{lyu} for closed subsets $Y_1, Y_2$ of $\Spec(B)$ such that $Y_1 \subset Y_2$.

\s\label{grd_Lyu} {\it Graded Lyubeznik functors}.

Let $A$ be a commutative Noetherian ring, and let $R=A[X_1, \ldots, X_n]$ be standard graded. We say $Y$ is a homogeneous closed subset of $\Spec(R)$ if $Y=V(f_1, \ldots, f_s)$, where $f_i$'s are homogeneous polynomials in $R$. We say $Y$ is a homogeneous locally closed subset of $\Spec(R)$ if $Y=Y''-Y'$, where $Y'\subset Y''$ are homogeneous closed subsets of $\Spec(R)$. Let $^*\Mod(R)$ denote the category of graded $R$-modules. Then we have an exact sequence of functors on $^*\Mod(R)$, 
\begin{equation}\label{glyu}
\cdots \to H_{Y'}^i(-) \to H_{Y''}^i(-) \to H_Y^i(-) \to H_{Y'}^{i+1}(-) \to \cdots.
\end{equation}
\begin{definition}\label{grd_Lyu_def}
A {\it graded Lyubeznik functor} $\TT$ is a composite functor of the form $\TT=\TT_1 \circ \TT_2 \circ \cdots \circ\TT_r$ where each $\TT_j$ is either $H_{Y_j}^i(-)$ for some homogeneous locally closed subset $Y_j$ of $\Spec(R)$ or the kernel, image, or cokernel of any arrow appearing in \eqref{glyu} with $Y=Y_j, ~Y'=Y'_j$ and $Y''=Y''_j$ such that $Y_j=Y''_j-Y'_j$ and $Y'_j \subset Y''_j$ are homogeneous closed subsets of $\Spec(R)$.
\end{definition}

\s {\it Graded Lyubeznik functor under flat maps}.\label{fglyu}

Let $\phi: A \to C$ be a flat homomorphism of Noetherian rings, and let $R=A[X_1, \ldots, X_n]$ be standard graded. We set $S=R\otimes_A C= C[X_1, \ldots, X_n]$. Clearly $\phi$ induces a map $\phi':R \to S$. Let $\TT$ be a graded Lyubeznik functor on $^* \Mod(R)$. Let $Y=V(f_1, \ldots, f_s)$ be a homogeneous closed subset of $\Spec(R)$.
Set $Y'=V(f'_1, \ldots, f'_s) $, where $f'_i= \phi'(f_i)$. Note that $Y'$ is a homogeneous closed subset of $\Spec(S)$. 
Clearly we have a homogeneous isomorphism $H_Y^i(-) \otimes_A C \cong H_{Y'}^i(-)$. If $Y$ is a homogeneous locally closed subset of $\Spec(R)$, say $Y= Y_2-Y_1$, then we put $Y'=Y'_2-Y'_1$. Thus $Y'$ is a homogeneous locally closed subset of $\Spec(S)$. 
Applying the functor $(-) \otimes_A C$ to \eqref{glyu}, 
we get 
a homogeneous isomorphism $H_{Y'}^i(-) \cong H_Y^i(-) \otimes_A C$. More generally, if $\TT$ is a graded Lyubeznik functor on $^*\Mod(R)$, then $\TT \otimes_A C$ is a graded Lyubeznik functor on $^*\Mod(S)$.
\begin{note*}\label{flat-L}
	The flat maps, we are interested in are the following:
	
	(i) $\phi: A \to A_W$, where $W$ is a multiplicatively closed subset of $A$, 
	
	(ii) 
$\psi: A \to A[[Y]]_Y$, where $A[[Y]]$ is a formal power series ring over $A$ in one variable. This is an important ingredient in our proofs of Theorems 8.5 and 8.6.
 
\end{note*}

\section{Generalized Eulerian \texorpdfstring{$A_m(A)$}{Am(A)}-modules}
Let $A$ be a commutative Noetherian ring containing a field $K$ of characteristic zero. Let $A_m(A)$ be the $m^{th}$ Weyl algebra over $A$. Notice that
\[A_m(A)= A \otimes_K A_m(K)= A\langle X_1, \ldots, X_m, \partial_1, \ldots, \partial_m\rangle.\] 
We can consider $A_m(A)$ graded, by giving $\deg z=0$ for all $z \in A$, $\deg X_i=1,~\deg \partial_i=-1$. The Euler operator on $A_m(A)$, denoted by $\e_m$, is defined as $\e_m:= \sum_{i=1}^m X_i\partial_i$. Note that $\deg \e_m=0$. Let $M$ be a graded $A_m(A)$-module. For any homogeneous element $y$ of $M$, we set $|y|= \deg y$.

\begin{definition}
A graded $A_m(A)$-module $M$ is said to be {\it generalized Eulerian} if for each homogeneous element $y$ of $M$ there exists a positive integer $a$ {\rm(}depending on $y${\rm)} such that \[(\e_m-|y|)^a \cdot y=0. \]
\end{definition}

We now recall the following well known result and outline a proof.
\begin{lemma}\label{noether}
If $\Lambda$ is left {\rm(}resp. right{\rm)} Noetherian {\rm(}not necessarily commutative{\rm)} ring then the Weyl Algebra $A_m(\Lambda)$ is also left {\rm(}resp. right{\rm)} Noetherian for all $m \geq 1$.
\end{lemma}

\begin{proof}
As $\Lambda$ is left 
Noetherian so by \cite[Theorem 9]{Mc}, a non-commutative version of Hilbert basis theorem, we get that $A_1(\Lambda)$ is left Noetherian. 
Since $A_m(\Lambda)=A_1(A_{m-1}(\Lambda))$, the result follows by induction.
\end{proof}

We recall the following result from \cite{TP2} which says that the class of generalized Eulerian modules is closed under extension.

\begin{lemma}\cite[Proposition 3.11]{TP2}\label{sesge}
Let $0 \rt M_1 \rt M_2 \rt M_3 \rt 0$ be a short exact sequence of $A_m(A)$-modules. Then $M_2$ is generalized Eulerian if and only if $M_1$ and $M_3$ are generalized Eulerian. 
\end{lemma}

We first develop some basic properties that are known for $A_m(K)$-modules. 
All these features will be used extensively in Section 7 to obtain the rigidity property.
To prove \cite[Proposition 3.2, Proposition 3.5]{TP6}  and \cite[Proposition 5.3, Proposition 5.5]{TPJS}, the authors nowhere used the fact that $K$ is a field. Therefore, we obtain our result using the same line of proofs given in \cite{TP6} and \cite{TPJS}. 
To provide the readers some idea, we repeat the 
proof of Proposition \ref{vc}.
\begin{proposition}\label{pe}
Let $M$ be a generalized Eulerian $A_m(A)$-module. Then the $A_{m-1}(A)$-module $H_\nu(\partial_m; M)(-1)$ is generalized Eulerian for $\nu=0, 1$.
\end{proposition}

\begin{proposition}\label{ve}
Let $M$ be a generalized Eulerian $A_m(A)$-module. Then the $A_{m-1}(A)$-module $H_\nu(X_m; M)$ is generalized Eulerian for $\nu=0, 1$.	
\end{proposition}

\begin{proposition}\label{pc}
Let $M$ be a generalized Eulerian $A_1(A)$-module. Then for $\nu=0, 1$, the module $H_\nu(\partial_1; M)$ is concentrated in degree $-1$, i.e., $H_\nu(\partial_1; M)_n=0$ for $n \neq -1$.
\end{proposition}

\begin{proposition}\label{vc}
Let $M$ be a generalized Eulerian $A_1(A)$-module. Then for $\nu=0, 1$, the module $H_\nu(X_1; M)$ is concentrated in degree $0$, i.e., $H_\nu(X_1; M)_n=0$ for $n \neq 0$.
\end{proposition}

\begin{proof}
As the map $M(-1) \overset{ \cdot X_1}{\lrt} M$ is $A$-linear so $H_i(X_1; M)$ are $A$-modules for $i=0, 1$. Thus we have an exact sequence of $A$-modules 
\[0 \rt H_1(X_1; M) \rt M(-1) \overset{\cdot X_1}{\lrt} M \rt H_0(X_1; M)\rt 0.\]
Let $\eta \in H_1(X_1; M)(1) \subseteq M$ be non-zero and homogeneous of degree $u$. Since $M$ is generalized Eulerian, we have 
\[(X_1 \partial_1-u)^b \eta=0 \quad \text{for some } b \geq 1.\]
Using the relation $\partial_1X_1-X_1 \partial_1= 1$, we can write 
\[0= (\partial_1X_1-(u+1))^b\eta= (*)X_1\eta + (-1)^b(u+1)^b\eta,\] 
where by $(*)$ we mean $\sum_{i=0}^{b-1}(-1)^i\binom{b}{i}(u+1)^i(\partial_1X_1)^{b-i-1}\partial_1$. As $X_1 \eta=0$ and $\eta \neq 0$ so we get $u=-1$, that is, $\eta \in H_1(X_1; M)(1)_{-1}= H_1(X_1; M)_0$. Hence $H_1(X_1; M)$ is concentrated in degree $0$.
	
Let $\eta' \in H_0(X_1; M)$ be non-zero and homogeneous of degree $v$. Therefore $\eta' \in (M/X_1M)_v$ and hence $\eta'= \beta+X_1M$ for some $\beta \in M_v$. Since $M$ is generalized Eulerian, we have 
\[(X_1 \partial_1-v)^a \beta=0 \quad \text{for some } a \geq 1.\] 
Notice that $(X_1 \partial_1-v)^a= X_1 \cdot (\star)+(-1)^av^a$. Thus 
\[X_1 \cdot (*)\beta+(-1)^av^a\beta=0,\] 
where $(\star)$ stands for $\sum_{j=0}^{a-1}\binom{a}{i}(-1)^iv^i \partial_1(X_1\partial_1)^{a-j-1}$. Going mod $X_1M$ we get $(-1)^av^a\eta'=0$. As $\eta' \neq 0$ so $v=0$, i.e., $\eta' \in H_0(X_1; M)_0$. Hence $H_0(X_1; M)$ is concentrated in degree $0$.
\end{proof}

\section{Some properties of \texorpdfstring{$A_m(A)$}{Am(A)}-modules}
Throughout this section, we write $X:=X_1, \ldots, X_m$ and $\partial:=\partial_1, \ldots, \partial_m$.

The following result is similar to \cite[Theorem 6.2]{BJ}. Note that in \cite[Theorem 6.2]{BJ}, $M$ is assumed to be holonomic. 

\begin{lemma}\label{surjective}
	Let $M$ be an $A_m(A)$-module. Set $N_1= \{y \in M\mid\partial_m^t y=0$ for some $t\geq 1\}$ and $N_2= \{y \in M\mid X_m^t y=0$ for some $t\geq 1\}$. Then 
	\begin{enumerate}[\rm (1)]
		\item $N_1$ is an $A_m(A)$-submodule of $M$.
		\item $N_1= \partial_m N_1$.
		\item $N_2$ is an $A_m(A)$-submodule of $M$.
		\item $N_2= X_m N_2$.
	\end{enumerate}
\end{lemma}

\begin{proof}
(1) Let $y \in N_1$. Then $\partial_m^sy =0$ for some $s \geq 1$. Since $\partial_i \partial_j=\partial_j\partial_i$ for all $i, j$ and $X_i \partial_j= \partial_j X_i$ for all $i \neq j$, we get that $\partial_m^s\partial_iy = \partial_i\partial_m^sy =0$ for all $i$ and $\partial_m^sX_iy = X_i\partial_m^sy =0$ for all $i \neq m$. Thus $\partial_iy \in N_1$ for all $i$, and $X_iy \in N_1$ for all $i \neq m$. Moreover, it is easy to check that $\partial_m^\alpha X_m= X_m \partial_m^\alpha+\alpha\partial_m^{\alpha-1}$ for all $\alpha \geq 1$. Therefore, 
\[\partial_m^{s+1}X_my= X_m \partial_m^{s+1}y+(s+1)\partial_m^sy=0,\] 
and hence $X_my \in N_1$. It follows that $N_1$ is an $A_m(A)$-submodule of $M$.
	
(2) Clearly $\partial_m N_1 \subseteq N_1$. Let $y \in N_1$. Then $\partial_m^sy =0$ for some $s \geq 1$. We use induction on $s$ to show $N_1 \subseteq \partial_m N_1$. If $s=1$, then $\partial_m(X_my)=X_m(\partial_m y)+ y= y$. As $X_m N_1 \subseteq N_1$ so we get $y \in \partial_m N_1$. We now assume that $s \geq 2$, and the result is true for $s-1$. If $\partial_m^sy =0$, then $\partial_m^{s}X_my =X_m \partial_m^{s}y+s\partial_m^{s-1}y=s\partial_m^{s-1}y$. Thus
\[\partial_m^{s}X_my-s\partial_m^{s-1}y=\partial_m^{s-1}(\partial_mX_my-sy)=0.\] 
So $\partial_mX_my-sy \in N_1$, and hence by induction hypothesis $\partial_mX_my-sy \in \partial_m N_1$. Since $\partial_mX_my \in \partial_m N_1$, we get $sy \in \partial_m N_1$. Therefore, $y \in \partial_m N_1$ and the result follows. 

\vspace{0.15cm}
%
We skip the proofs of (3) and (4) since they can be obtained following 
the same line of proofs given for (1) and (2), respectively. 
\end{proof}

\s \label{Fou_trans}{\it Fourier transform:} Let $\FF: A_m(A) \to A_m(A)$ be an automorphism 
defined by $\FF(\partial_i)=-X_i$ and $\FF(X_i)= \partial_i$ for $i=1, \ldots, m$. Clearly $\FF(\underline{X}\underline{\partial})=\FF(\underline{X})\FF(\underline{\partial})$, where $\underline{X}:=X_1 \cdots X_m$ and $\underline{\partial}:=\partial_1 \cdots \partial_m$. Notice that \[\FF(1)= \FF(\partial_iX_i-X_i\partial_i)= \FF(\partial_i)\FF(X_i)-\FF(X_i)\FF(\partial_i)= -X_i\partial_i+\partial_i X_i= 1,\] 
and the restriction map $\FF|_R: R:=A[X_1, \ldots, X_m] \to S:=A[\partial_1, \ldots, \partial_m]$
is an isomorphism. For any left $A_m(A)$-module $M$, let $M^\FF$ denote a new module such that (i) $M^\FF=M$ as abelian groups, and (ii) the action of $R$ on $M^\FF$ is defined by 
\[r \cdot y= \FF(r)y \quad \mbox{ for all } r \in R \mbox{ and } y \in M^\FF.\] 
The new module $M^{\FF}$ is called the {\it Fourier transform} of $M$. Similarly we define 
\[s*z= \FF^{-1}(s)z \quad \mbox{ for all } s \in S\mbox{ and } z \in M^{\FF^{-1}}.\] 
Notice that $r \cdot y= \FF(r) * y=\FF^{-1}\left(\FF(r)\right)y=ry$ for all $r \in R$, and all $y \in (M^{\FF^{-1}})^\FF$. It follows that 
$(M^{\FF^{-1}})^\FF=M$. Similarly, one can check that $(M^\FF)^{\FF^{-1}}= M$. 

Let $B$ be a commutative Noetherian ring, $I$ be an ideal in $B$, and let $L$ be a $B$-module. Set $\Gamma_I(L)= \bigcup_{t \in \N} (0:_L I^t)$, the submodule of $L$ consisting of all elements of $L$ which are annihilated by some power of $I$.
\begin{lemma}\label{sub-gamma}
Let $A$ be a commutative Noetherian ring containing a field of characteristic zero. Let $R=A[X_1, \ldots, X_m]$ and $S=A[\partial_1, \ldots, \partial_m]$. Let $I$ and $J$ be homogeneous ideals in $R$ and $S$ respectively. If $M$ is an $A_m(A)$-module, then $\Gamma_I(M)$ and $\Gamma_J(M)$ are graded $A_m(A)$-modules.
\end{lemma}

\begin{proof}
Notice that $R$ and $S$ are commutative rings contained in the left Noetherian ring $A_m(A)$. Let $I$ be a homogeneous ideal in $R$.
	
\begin{claim*}
For any $y \in M$,
\[I^{t+1} \partial_iy \subseteq \partial_i I^{t+1}y+I^ty \quad \mbox{ for } i=1, \ldots, m \mbox{ and all } t \geq 0.\]
\end{claim*}
	
We know that $I^{t+1}=\langle g_1 \cdots g_{t+1} \mid g_j \in I\rangle$. Fix $i$ and set $\underline{g}:=g_1 \cdots g_{t+1}$. Since $\partial_i\underline{g}= \underline{g}\partial_i+\partial_i(\underline{g})$, 
we have $\underline{g}\partial_iy= \partial_i\underline{g}y- \partial_i(\underline{g})y$.
Using the chain rule, we get that 
\[\partial_i(\underline{g})=\partial_i(g_1 \cdots g_{t+1})= \sum_{j=1}^{t+1}g_1 \cdots\partial_i(g_j)\cdots g_{t+1}.\] 
Thus 
$\underline{g}\partial_iy \in \partial_i I^{t+1}y+I^ty$. The claim follows.
	
Let $y \in \Gamma_I(M)$. Then $I^sy=0$ for some $s \geq 0$. From the above claim we get $I^{s+1}\partial_iy \subseteq \partial_i I^{s+1}y+I^{s}y=0$. Therefore $\partial_iy \in \Gamma_I(M)$ and hence $\partial_i \Gamma_I(M) \subseteq \Gamma_I(M)$ for all $i$. Moreover, we have $A\Gamma_I(M) \subseteq \Gamma_I(M)$ and $X_i \Gamma_I(M) \subseteq \Gamma_I(M)$ for all $i$. It follows that $\Gamma_I(M)$ is an $A_m(A)$-submodule of $M$.
	
Note that $\Gamma_{\FF(I)}(M^\FF)= \{y \in M^\FF\mid\FF(I)^sy=0 \text{ for some } s \geq 0\}$. As $\FF$ is a ring homomorphism so we get that $\FF(I)^ty= \FF(I^t)y= I^t \cdot y$ for all $t\geq 0$. Therefore, $\Gamma_I(M)^\FF= \Gamma_{\FF(I)}(M^\FF)$. Let $J$ be a homogeneous ideal in $S$. Set $I:=\FF^{-1}(J)$. It is easy to check that $I$ is a homogeneous ideal in $R$, and $\FF(I)= \FF(\FF^{-1}(J))= J$. Since $M^{\FF}$ is an $A_m(A)$-module, by the above observation we have $\Gamma_I(M^{\FF^{-1}}) \subseteq_{A_m(A)} M^{\FF^{-1}}$. Applying $\FF$ on both sides, we get that $\Gamma_I(M^{\FF^{-1}})^\FF \subseteq_{A_m(A)} (M^{\FF^{-1}})^\FF$, 
i.e., $\Gamma_{\FF(I)}\left((M^{\FF^{-1}})^\FF\right) \subseteq_{A_m(A)} M$ and hence $\Gamma_{J}(M) \subseteq_{A_m(A)} M$.
\end{proof}

\begin{remark}\label{Gsub} Since $(X_1, \ldots, X_m)$ and $(\partial_1, \ldots, \partial_m)$ are homogeneous ideals in $R$ and $S$ respectively, by the above lemma it follows that $\Gamma_{(X)}(M)$ and $\Gamma_{(\partial)}(M)$ are $A_m(A)$-modules.
\end{remark}

Now we prove the main result of this section.
\begin{proposition}\label{GTAM}
Let $M$ be a generalized Eulerian $A_m(A)$-module. Then 
\begin{enumerate}[\rm(1)]
\item $\Gamma_{(X_1, \ldots, X_m)}(M)_n=0$ for all $n \geq -m+1$.
\item $\Gamma_{(\partial_1, \ldots, \partial_m)}(M)_n=0$ for all $n \leq -m$.
\end{enumerate}
\end{proposition}

\begin{proof}
(1) We proceed by induction on $m$. Let $m=1$. From Lemma \ref{surjective} we have $\Gamma_{(X_1)}(M)= X_1 \Gamma_{(X_1)}(M)$, i.e., the map $\Gamma_{(X_1)}(M)(-1) \overset{ \cdot X_1}{\lrt} \Gamma_{(X_1)}(M)$ is surjective. Thus we get a short exact sequence 
\begin{equation}\label{g1}
0 \to L \to \Gamma_{(X_1)}(M)(-1) \overset{\cdot X_1}{\lrt} \Gamma_{(X_1)}(M) \to 0.
\end{equation}
Since $M$ is a generalized Eulerian $A_1(A)$-module, by Lemma \ref{sesge} and Remark \ref{Gsub},
$\Gamma_{(X_1)}(M)$ is a generalized Eulerian $A_1(A)$-module. Therefore, $L= H_1(X_1; \Gamma_{(X_1)}(M))$ is concentrated in degree $0$ by Proposition \ref{vc}. From \eqref{g1} it follows that $0 \to \Gamma_{(X_1)}(M)_{n-1} \overset{\cdot X_1}{\lrt} \Gamma_{(X_1)}(M)_n$ for all $n \neq 0$. Thus 
\[\Gamma_{(X_1)}(M)_0 \overset{\cdot X_1}{\hookrightarrow} \Gamma_{(X_1)}(M)_1\overset{\cdot X_1}{\hookrightarrow} \Gamma_{(X_1)}(M)_2 \overset{\cdot X_1}{\hookrightarrow} \Gamma_{(X_1)}(M)_3 \overset{\cdot X_1}{\hookrightarrow} \cdots.\] 
Let $y \in \Gamma_{(X_1)}(M)_n$ for some $n \geq 0$. Then $X_1^t y=0$ for some $t \geq 1$. As $\Gamma_{(X_1)}(M)_n \hookrightarrow \Gamma_{(X_1)}(M)_{n+t}$ so we get $y=0$ and hence $\Gamma_{(X_1)}(M)_n=0$ for all $n \geq 0$.
	
We now assume the result is true for $m-1$. Since $\Gamma_{(X)}(M) \subseteq_{A_m(A)} M$ by Lemma \ref{sub-gamma}, we have $\Gamma_{(X)}(M)$ is generalized Eulerian $A_m(A)$-module. Set $N= \{y \in \Gamma_{(X)}(M)\mid X_m^t y=0 ~\text{for some } t\geq 1\}$. 
Clearly $N= \Gamma_{(X)}(M)$. 
Now $N= X_m N$ by Lemma \ref{surjective}(1), and hence the map $\Gamma_{(X)}(M)(-1)\overset{\cdot X_m}{\lrt} \Gamma_{(X)}(M)$ is surjective. Thus we get a short exact sequence 
\begin{equation}\label{g2}
0 \to V \to \Gamma_{(X)}(M)(-1) \overset{X_m}{\lrt} \Gamma_{(X)}(M) \to 0.
\end{equation}
By Proposition \ref{ve} we have $V= H_1(X_m; \Gamma_{(X)}(M))$ is a generalized Eulerian $A_{m-1}(A)$-module. Notice that $V= \Gamma_{(X_1, \ldots, X_{m-1})}(V)$. By induction hypothesis it follows that $V_n=0$ for all $n \geq -(m-1)+1=-m+2$. From \eqref{g2} we get 
\[\Gamma_{(X)}(M)_{-m+1} \overset{\cdot X_m}{\hookrightarrow} \Gamma_{(X)}(M)_{-m+2}\overset{\cdot X_m}{\hookrightarrow} \Gamma_{(X)}(M)_{-m+3} \overset{\cdot X_m}{\hookrightarrow} \cdots.\]
Let $y \in \Gamma_{(X)}(M)_{-m+w}$ for some $w \geq 1$. Then $X_m^t y=0$ for some $t\geq1$. From the map $\Gamma_{(X)}(M)_{-m+w} \hookrightarrow \Gamma_{(X)}(M)_{-m+w+t}$, we get $y=0$ and hence $\Gamma_{(X)}(M)_n=0$ for all $n \geq -m+1$.

The proof of (2) is similar to (1). 
One needs 
to use Lemma \ref{surjective}, Propositions \ref{pe}, and \ref{pc}.
\end{proof}

\section{Countable generation and local cohomology modules}

Throughout this section, $B$ is a commutative Noetherian ring. 
We discuss some properties of a countably generated module, which we will use in the 
following section. We further show that the local cohomology module $H^i_I(B)$ is a countably generated $B$-module for any ideal $I \subseteq B$ and each $i \geq 0$.

\begin{definition}
We say a $B$-module $L$ is countably generated if there exists a countable set of elements which generates $L$ as a $B$-module.
\end{definition}

The following result is well-known, but we do not have a reference.
\begin{lemma}\label{scg}
Let $0 \to L_1 \overset{f}{\lrt} L_2 \overset{g}{\lrt} L_3 \to 0$ be a short exact sequence of $B$-modules. Then $L_2$ is countably generated if and only if $L_1$ and $L_3$ are countably generated.
\end{lemma}

\begin{proof}
We may assume that $f$ is inclusion and $g$ is the quotient map. Let $L_2$ be countably generated. So we have $L_2= \bigcup_{i \in \N} D_i$, where $D_i$'s are finitely generated $B$-modules and $D_0 \subseteq D_1 \subseteq D_2 \cdots$. Since $f(L_1)=L_1$ is a submodule of $L_2$, we get $L_1= L_2 \cap L_1= \bigcup_{i \in \N} (D_i \cap L_1)$. Moreover, $B$ is a Noetherian ring and $D_i \cap L_1$ is a submodule of $D_i$. Therefore $D_i \cap L_1$ is finitely generated and hence $L_1$ is countably generated. As $L_3 \cong L_2/L_1$ as $B$-modules so we get that $L_3$ is countably generated with same generating set as of $L_2$.
	
Conversely, let $L_1$ and $L_3$ be generated by $\{\alpha_i\}_{i \geq 1}$ and $\{\gamma_i\}_{i \geq 1}$ respectively. Let $\beta_i \in L_2$ be such that $g(\beta_i)=\overline{\beta_i}= \gamma_i$. 
	
\begin{claim*}
$L_2$ is generated by $\{\alpha_1, \ldots, \beta_1, \ldots\}$.
\end{claim*} 
	
Let $z \in L_2$. Then $g(z)= \sum b_i \gamma_i$ where $b_i=0$ for all but finitely many $i$. Set $z^*=\sum b_i \beta_i$. Then $g(z)=g(z^*)$ and hence $z-z^* \in \ker g = L_1$. So $z-z^*= \sum a_i \alpha_i$ where $a_i=0$ for all but finitely many $i$. Thus $z= \sum b_i \beta_i+ \sum a_i \alpha_i$. The claim follows. 
\end{proof}

\s \label{van_tam_sa} Let $A$ be a commutative Noetherian ring containing a field $K$ of characteristic zero. Let $R=A[X_1, \ldots, X_m]$ and $S=A[\partial_1, \ldots, \partial_m]$ be polynomial rings. Let $A_m(A)$ be the $m^{th}$ Weyl algebra over $A$. We consider $R$, $S$, and $A_m(A)$ as graded with $\deg z=0$ for all $z \in A$, $\deg X_i=1$ and $\deg \partial_i=-1$.
\begin{lemma}\label{vpcg}
Let $M$ be an $A_m(A)$-module. The following assertions are equivalent:
\begin{enumerate}[\rm (1)]
\item $M$ is a countably generated $A_m(A)$-module.
\item $M$ is a countably generated $R$-module.
\item $M$ is a countably generated $S$-module.
\end{enumerate}
\end{lemma}  

\begin{proof} 
Notice that $R$ and $S$ are sub-rings of $A_m(A)$. So the implications $(2) \implies (1)$ and $(3) \implies (1)$ follow trivially. We now prove $(1) \implies (3)$. 
	
Set $D:=A_m(A)$. Let $M$ be a countably generated $D$-module. Therefore $M= \bigcup_{i \in \N} M_i$ where $M_i$'s are finitely generated $D$-modules, and $M_0 \subseteq M_1 \subseteq M_2 \cdots$.  Since $M_i$ is a finitely generated $D$-module, there exists a surjctive map $D^{t_i} \to M_i \to 0$ for some $t_i \geq 0$. Moreover, $D=\bigcup_{i \in \N} D_i$ as a $S$-module, where $D_i$ is a left (resp. right) $S$-module generated by all monomials in $X_1, \ldots, X_m$ of degree less than or equal to $i$. Clearly $D_i$'s are finitely generated left (resp. right) $S$-modules, and $D_0 \subseteq D_1 \subseteq D_2 \cdots$. Thus $D$ is a countably generated left (resp. right) $S$-module and so is $D^{t_i}$. By Lemma \ref{scg} we get that $M_i$ is a countably generated $S$-module for all $i$. Hence $M$ is a countably generated $S$-module. 
	
To prove the implication $(1) \implies (2)$, we only need to replace `$S$' by `$R$' and `$X_i$' by `$\partial_i$' in the above proof. 
\end{proof}	

The following result is 
well-known. 
Due to the lack of reference, we provide a proof here.

\begin{lemma}\label{cg}
	Let ${\bf Y}: 0 \to Y^0 \overset{d^0}{\lrt} Y^1 \overset{d^1}{\lrt} Y^2 \overset{d^2}{\lrt} \cdots$ be a co-chain complex of countably generated $B$-modules. Then $H^i({\bf Y})$ is countably generated for all $i$.
\end{lemma}

\begin{proof}
	Notice that $\ker d^i$ is a submodule of $Y^i$. Since $Y^i$ is countably generated, by Lemma \ref{scg} we get that $\ker d^i$ is countably generated. From Lemma \ref{scg}, it follows that $H^i({\bf Y})= \ker d^i/\im d^{i-1}$ is countably generated. 
\end{proof}

\begin{lemma}\label{cglc}
	Let $I$ be an ideal in $B$. Let $L$ be a countably generated $B$-module. Then $H_I^i(L)$ is countably generated.
\end{lemma}

\begin{proof}
	Let $I=(f_1, \ldots, f_s)$. Consider the \v Cech complex \[{\bf C}: 0 \to L \to \bigoplus_{i=1}^s L_{f_i} \to \cdots \to L_{f_1 \cdots f_s} \to 0.\] 
	
	\begin{claim*}
		If $N$ is a countably generated $B$-module, then so is $N_f$ 
		for any $f \in B$.
	\end{claim*} 
	
	Clearly $N_f= \bigcup_{i \geq 1} N_{1/f^i}$. Choose a countable generating set $W$ of $N$. Then $W_{1/f^i}:= \{z/f^i\mid z \in W\}$ is a generating set of $N_{1/f^i}$ as a $B$-module and hence $\bigcup_{i \geq 1} W_{1/f^i}$ is a generating set of $N_f$. As $W$ is a countable set, so is $W_{1/f^i}$. Thus $\bigcup_{i \geq 1} W_{1/f^i}$ is a countable set. It follows that $N_f$ is countably generated.
	
	Since $L$ is countably generated, ${\bf C}$ is a complex of countably generated $C$-modules. By Lemma \ref{cg}, it follows that $H^i_I(L)= H^i({\bf C})$ is countably generated.
\end{proof}

\begin{lemma}\label{lyucg}
	Let $\TT$ be a Lyubeznik functor on $\Mod(B)$. Then $\TT(L)$ is a countably generated $B$-module for any countably generated $B$-module $L$.
\end{lemma}

\begin{proof}
	Let $Y$ be any locally closed subset of $\Spec(B)$ with $Y=Y_2-Y_1$, where $Y_1 \subset Y_2$ are closed subsets of $\Spec(B)$. For any countably generated $B$-module $N$, we have an exact sequence 
	\[H_{Y_2}^i(N) \to H_Y^i(N) \to H_{Y_1}^{i+1}(N).\] 
	From Lemma \ref{cglc}, we get that $H_{Y_2}^i(N)$ and $H_{Y_1}^{i+1}(N)$ are countably generated. 
	By Lemma \ref{scg}, it follows that $H_Y^i(N)$ is countably generated. 
	
	Let $\TT=\TT_1 \circ \TT_2 \circ \cdots \circ \TT_r$. 
	By 
\ref{LyuDef}, $\TT_j(N)$ is countably generated for each $j$, and any countably generated $B$-module $N$. Notice that $\TT(L)= (\TT_1 \circ \TT_2 \circ \cdots\circ \TT_r)(L)=\TT_1(\cdots(\TT_r(L))\cdots)$. Hence $\TT(L)$ is a countably generated $B$-module. 
\end{proof}

\section{Vanishing and tameness}
In this section, we discuss the vanishing and tameness properties of generalized Eulerian modules.

The following result is well-known, see \cite{Cla}. 
\begin{lemma}\label{coun-un}
A nonzero finite dimensional vector space over an uncountable field is not a union of countably many proper subspaces.
\end{lemma}

The following result from \cite{TP4} is important for our study.
\begin{lemma}\label{cass}\cite[Lemma 2.3]{TP4}
Let $L$ be a countably generated $B$-module. Then $\Ass_B L$ is a countable set. 
\end{lemma}

We also need the following proposition from \cite{TP2} to prove the forthcoming result.
\begin{proposition}\label{asspre}\cite[Proposition 12.1]{TP2}
Let $f: B \to C$ be a homomorphism of commutative Noetherian rings. Let $L$ be a $C$-module. Then \[\Ass_B L= \{P \cap B~|~P \in \Ass_C L\}.\]
In particular, if $\Ass_C L$ is a finite set, then so is $\Ass_B L$.
\end{proposition}

We are now in a position to establish the following result under the extra hypothesis that $K$ is uncountable.
\begin{proposition}[with hypothesis as in \ref{van_tam_sa}]\label{inj_map}
Further, assume $K$ is uncountable.
Let $M= \bigoplus_{n \in \Z}M_n$ be a countably generated, 
generalized Eulerian $A_m(A)$-module. Then there exists some homogeneous element $\eta \in K[X_1, \ldots, X_m]$ of degree $1$ such that $M_n \overset{\cdot\eta}{\rt} M_{n+1}$ is an injective map for all $n \geq -m+1$, and some homogeneous element $\xi \in K[\partial_1, \ldots, \partial_m]$ of degree $-1$ such that $M_n \overset{\cdot\xi}{\rt} M_{n-1}$ is an injective map for all $n \leq -m$.
\end{proposition}

\begin{proof}
By Proposition \ref{GTAM} we have $\Gamma_{(X)}(M)_n=0$ for all $n \geq -m+1$. 
From the short exact sequence $0 \to \Gamma_{(X)}(M) \to M \to M/\Gamma_{(X)}(M) \to 0$, we get that 
$M_n= \left(M/\Gamma_{(X)}(M)\right)_n$ for all $n \geq -m+1$. As $R$ is Noetherian and $(X_1, \ldots, X_m)$ is an ideal in $R$ so by \cite[Lemma 6.2]{TP3} we get 
\begin{equation}\label{ass-QX-LC}
\Ass_{R} \frac{M}{\Gamma_{(X)}(M)}= \{P \in \Ass_{R} M \mid P \nsupseteq (X)\}.
\end{equation}
Set $\overline{M}= M/\Gamma_{(X)}(M)$. Note that if $P \in \Ass_R M$, then $P$ is homogeneous. By Lemmas \ref{vpcg} and \ref{scg}, $\overline{M}$ is a countably generated $R$-module. Therefore, $\Ass_{R} \overline{M}$ is a countable set by Lemma \ref{cass}. Since $C:= K[X_1, \ldots, X_m] \subset R$ is a sub-ring, by Proposition \ref{asspre} we have $\Ass_C \overline{M}= \Ass_R \overline{M} \cap C$. So $\Ass_C \overline{M}$ is also a countable set. Additionally, for each $Q \in \Ass_C \overline{M}$, there exists some $P \in \Ass_R \overline{M}$ such that $Q= P \cap C$.
By \eqref{ass-QX-LC}, $P \nsupseteq (X)$ for any $P \in \Ass_R \overline{M}$. Hence $Q \cap C_1 = P \cap C_1 \subsetneq C_1$. As $\overline{M}$ is a graded $C$-module so $Q$ is homogeneous.
We denote $H:= \bigcup_{Q \in \Ass_C \overline{M}} Q \cap C_1$.
Since $K$ is uncountable, $H \subsetneq C_1$ by Lemma \ref{coun-un}. Thus there exists some $\eta \in C_1-H$. Note that $\eta:= b_1 X_1+ \cdots+b_m X_m \in K[X_1, \cdots, X_m]$ is a non-zero divisor on $\overline{M}$, where $b_i \in K$. Consequently, $M_n \overset{\cdot\eta}{\rt} M_{n+1}$ is an injective map for all $n \geq -m+1$.
	
\vspace{0.15cm}
By Proposition \ref{GTAM} we have $\Gamma_{(\partial)}(M)_n=0$ for all $n \leq -m$. From the short exact sequence $0 \to \Gamma_{(\partial)}(M) \to M \to M/\Gamma_{(\partial)}(M) \to 0$, we get that $M_n= \left(M/\Gamma_{(\partial)}(M)\right)_n$ for all $n \leq -m$. As $S$ is Noetherian and $(\partial_1, \ldots, \partial_m)$ is an ideal in $S$ so by \cite[Lemma 6.2]{TP3} we get
\begin{equation}\label{ass-Qpartial-LC}
\Ass_{S} \frac{M}{\Gamma_{(\partial)}(M)}= \{P' \in \Ass_{S} M \mid P' \nsupseteq (\partial)\}.
\end{equation}
Set $\widetilde{M}= M/\Gamma_{(\partial)}(M)$. Note that if $P' \in \Ass_S M$, then $P'$ is homogeneous. Since $M$ is a countably generated $A_m(A)$-module, by Lemma \ref{vpcg}, it is a countably generated $S$-module. From Lemma \ref{scg} it follows that $\widetilde{M}$ is a countably generated $S$-module. Thus $\Ass_{S} \widetilde{M}$ is a countable set by Lemma \ref{cass}. As $T= K[\partial_1, \ldots, \partial_m] \subset S$ is a sub-ring so by Proposition \ref{asspre} we have $\Ass_T \widetilde{M}= \Ass_S \widetilde{M} \cap T$. Therefore, $\Ass_T \widetilde{M}$ is also a countable set. Moreover, for each $Q' \in \Ass_T \widetilde{M}$, there exists some $P' \in \Ass_S \widetilde{M}$ such that $Q'= P' \cap T$.
By \eqref{ass-Qpartial-LC}, $P' \nsupseteq (\partial)$ for any $P' \in \Ass_S \widetilde{M}$. Hence $Q' \cap T_{-1} = P' \cap T_{-1} \subsetneq T_{-1}$. As $\overline{M}$ is a graded $T$-module so $Q'$ is homogeneous.
We denote $H':= \bigcup_{Q' \in \Ass_T \widetilde{M}} Q' \cap T_{-1}$.
Since $K$ is uncountable, $H' \subsetneq T_{-1}$. Thus there exists some $\xi \in T_{-1}-H'$. Note that $\xi:= a_1 \partial_1+ \cdots+a_m \partial_m \in K[\partial_1, \ldots, \partial_m]$ is a non-zero divisor on $\widetilde{M}$, where $a_i \in K$. Consequently, $M_n \overset{\cdot\xi}{\rt} M_{n-1}$ is an injective map for all $n \leq -m$.
\end{proof}

As a consequence of Proposition \ref{inj_map} we get the following results.

\begin{corollary}[Vanishing]\label{cgvanish}
{\rm(}with hypothesis as in \ref{van_tam_sa}{\rm)}.
Further, assume K is uncountable.
Let $M= \bigoplus_{n \in \Z}M_n$ be a countably generated, 
generalized Eulerian $A_m(A)$-module. If $M \neq 0$, then $M_n \neq 0$ for infinitely many  $n\gg 0$ OR $M_n \neq 0$ for infinitely many $n \ll 0$. 
\end{corollary}

\begin{proof}
Suppose if possible there exists $r \geq -m+1$ such that $M_r \neq 0$ and $M_n =0$ for all $n > r$. By Proposition \ref{inj_map}, there exists some homogeneous element $\eta \in K[X_1, \ldots, X_m]$ of degree $1$ such that the map $M_n \overset{\cdot\eta}{\rt} M_{n+1}$ is injective for all $n \geq -m+1$. In particular, we have an injective map $M_{r} \overset{\cdot\eta}{\rt} M_{r+1}$. Thus $M_{r+1} \neq0$, a contradiction.
	
Next, let if possible there exists $s \leq -m$ such that $M_s \neq 0$ and $M_n =0$ for all $n <s$. By Proposition \ref{inj_map}, there exists some homogeneous element $\xi \in K[\partial_1, \ldots, \partial_m]$ of degree $-1$ such that $M_n \overset{\cdot\xi}{\rt} M_{n-1}$ is an injective map for all $n \leq -m$. In particular, we have an injective map $M_{s} \overset{\cdot\xi}{\rt} M_{s-1}$. Thus $M_{s-1} \neq 0$, a contradiction.
\end{proof}

\begin{corollary}[Tameness]\label{cgtame}
{\rm(}with hypothesis as in \ref{van_tam_sa}{\rm)}.
Further, assume $K$ is uncountable.
Let $M= \bigoplus_{n \in \Z}M_n$ be a countably generated, 
generalized Eulerian $A_m(A)$-module. Then
\begin{align*}
&M_{n_0} \neq 0 \text{ for some } n_0 \geq -m+1 \implies M_{n} \neq 0  \text{ for all } n \geq {n_0},\\
&M_{n_0} \neq 0 \text{ for some } n_0 \leq -m \implies M_{n} \neq 0  \text{ for all } n \leq {n_0}.
\end{align*}
\end{corollary}

\begin{proof}
By Proposition \ref{inj_map}, there exists some homogeneous element $\eta \in K[X_1, \ldots, X_m]$ of degree one such that $M_n \overset{\eta}{\rt} M_{n+1}$ is an injective map for all $n \geq -m+1$. In particular, for given $n_0 \geq -m+1$ we have an injective map $M_{n_0} \overset{\cdot\eta}{\rt} M_{n_0+1}$ and hence $M_{n_0+1} \neq 0$. Repeating this process we get $M_n \neq 0$ for all $n \geq {n_0}$.
	
Again by Proposition \ref{inj_map}, there exists some homogeneous element $\xi \in K[\partial_1, \ldots, \partial_m]$ of degree $-1$ such that $M_n \overset{\cdot\xi}{\rt} M_{n-1}$ is an injective map for all $n \leq -m$. In particular, for given $n_0 \leq -m$ we have an injective map $M_{n_0} \overset{\cdot \xi}{\rt} M_{n_0-1}$ and hence $M_{n_0-1} \neq 0$. 
Repeating this process we get $M_n \neq 0$ for all $n \leq {n_0}$.
\end{proof}

\section{Rigidity}
In this section, we 
extend the rigidity results, proved in \cite{TP2}, to the wider framework of generalized Eulerian $A_m(A)$-modules. To get \cite[Theorem 6.1, Theorem 6.2]{TP2}, the first author used the techniques of de Rham cohomology and Koszul homology of generalized Eulerian $A_m(K)$-modules. In Section 4, we have showed that
generalized Eulerian $A_m(A)$-modules also exhibit similar properties, see Propositions \ref{pe}, 
\ref{ve}, \ref{pc} and \ref{vc}.  
Due to this, Theorem \ref{rigid-graded} and Theorem \ref{rigidity} can be proved following exactly the same method used in \cite{TP2}. 
We duplicate the proofs here for convenience of the reader. We begin with the following result.
\begin{theorem}\label{rigid-graded} 
	Let $M= \bigoplus_{n \in \Z} M_n$ be a generalized Eulerian $A_m(A)$-module. Then 
	\begin{enumerate}[\rm(I)]
		\item \Tfae
		\begin{enumerate}[\rm(a)]
			\item $M_n \neq 0$ for infinitely many $n<0$.
			\item There exists $r$ such that $M_n \neq 0$ for all $n \leq r$.
			\item $M_n \neq 0$ for all $n \leq -m$.
			\item $M_s \neq 0$ for some $s \leq -m$.
		\end{enumerate}
		\item \Tfae
		\begin{enumerate}[\rm(a)]
			\item $M_n \neq 0$ for infinitely many $n \geq 0$.
			\item There exists  $s$ such that $M_n \neq 0$ for all $n \geq s$.
			\item $M_n \neq 0$ for all $n \geq 0$.
			\item $M_t \neq 0$ for some $t \geq 0$.
		\end{enumerate}
	\end{enumerate}
\end{theorem}
\begin{proof}
(I) The implications $(c) \implies (b) \implies (a) \implies (d)$ follow trivially. We only have to prove $(d) \implies (c)$. We will do this by induction on $m$. 
	
We first assume $m=1$. Consider the exact sequence 
\begin{equation}\label{tm1}
0 \rt H_1(\partial_1; M)_j \rt M_{j+1} \overset{\partial_1}{\lrt} M_j \rt H_0(\partial_1; M)_j\rt 0.
\end{equation}
Since $M$ is a generalized Eulerian $A_m(A)$-module so by Proposition \ref{pc} we get that $H_\nu(\partial_1; M)$ is concentrated in degree $-1$ for $\nu=0, 1$. Thus by exact sequence \eqref{tm1} it follows that \[M_j \cong M_{-1} \quad \text{for all } j \leq -2.\] Hence $M_j \cong M_{s} \neq 0$ for all $j \leq -1$. 
	
We now assume that $m \geq 2$ and the result is proved for $m-1$. Consider the exact sequence 
\begin{equation}\label{tm2}
0 \rt H_1(\partial_m; M)_j \rt M_{j+1} \overset{\partial_m}{\lrt} M_j \rt H_0(\partial_m; M)_j\rt 0.
\end{equation} 
By Proposition \ref{pe} we get $H_\nu(\partial_m; M)(-1)$ is a generalized Eulerian $A_{m-1}(A)$-module for $\nu=0, 1$. We consider the following three cases:

{\it Case 1:} $H_0(\partial_m; M)(-1)_j \neq 0$ for some $j \leq -m+1$.\\
By the induction hypothesis it follows that $H_0(\partial_m; M)(-1)_j \neq 0$ for all $j \leq -m+1$. Hence $H_0(\partial_m; M)_j \neq 0$ for all $j \leq -m$. So by exact sequence \eqref{tm2} we get that $M_j \neq 0$ for all $j \leq -m$.
	
{\it Case 2:} $H_1(\partial_m; M)(-1)_j \neq 0$ for some $j \leq -m+1$.\\
By the induction hypothesis it follows that $H_1(\partial_m; M)(-1)_j \neq 0$ for all $j \leq -m+1$. Hence $H_1(\partial_m; M)_j \neq 0$ for all $j \leq -m$. So by exact sequence \eqref{tm2} it follows that $M_{j+1} \neq 0$ for all $j \leq -m$ which implies $M_j \neq 0$ for all $j \leq -m+1$.
	
{\it Case 3:} $H_\nu(\partial_m; M)(-1)_j = 0$ for $\nu=0, 1$ and for ALL $j \leq -m+1$.\\
Then we have $H_\nu(\partial_m; M)_j = 0$ for $\nu=0, 1$ and for all $j \leq -m$. By exact sequence \eqref{tm2} it follows that $M_j \cong M_{-m+1}$ for all $j \leq -m+1$. Hence $M_j \cong M_s \neq 0$ for all $j \leq -m$.
	
	\vspace{0.15cm}
	(II) Clearly $(c) \implies (b) \implies (a) \implies (d)$. We only have to prove $(d) \implies (c)$. To do this we use induction on $m$. 
	
	We first assume $m=1$. Consider the exact sequence 
	\begin{equation}\label{tm3}
	0 \rt H_1(X_1; M)_j \rt M_{j-1} \overset{X_1}{\lrt} M_j \rt H_0(X_1; M)_j\rt 0.
	\end{equation}
	Since $M$ is a generalized Eulerian $A_m(A)$-module so by Proposition \ref{vc} we get that $H_\nu(X_1; M)$ is concentrated in degree $0$ for $\nu=0, 1$. Thus by the exact sequence \eqref{tm3} it follows that \[M_j \cong M_0 \quad \text{for all } j \geq 0.\] Hence $M_j \cong M_{t} \neq 0$ for all $j \geq  0$. 
	
	We now assume that $m \geq 2$ and the result is proved for $m-1$. Consider the exact sequence
	\begin{equation}\label{t2}
	0 \rt H_1(X_m; M)_j \rt M_{j-1} \overset{X_m}{\lrt} M_j \rt H_0(X_m; M)_j\rt 0.
	\end{equation}
	By Proposition \ref{ve} we have $H_\nu(X_m; M)$ is a generalized Eulerian $A_{m-1}(A)$-module. We consider the following three cases:
	
	{\it Case 1:} $H_0(X_m; M)_j \neq 0$ for some $j \geq 0$.\\
	By the induction hypothesis we get $H_0(X_m; M)_j \neq 0$ for all $j \geq 0$. So by exact sequence \eqref{t2} it follows that $M_j \neq 0$ for all $j \geq 0$.
	
	{\it Case 2:} $H_1(X_m; M)_j \neq 0$ for some $j \geq 0$.\\
	By the induction hypothesis we get $H_1(X_m; M)_j \neq 0$ for all $j \geq 0$. So by exact sequence \eqref{t2} it follows that $M_{j-1} \neq 0$ for all $j \geq 0$ and hence $M_j \neq 0$ for all $j \geq-1$. 
	
	{\it Case 3:} $H_\nu(X_m; M)_j = 0$ for $\nu=0, 1$ and for ALL $j \geq 0$.\\
	By exact sequence \eqref{t2} we get $M_j \cong M_{-1}$ for all $j \geq 0$. Hence $M_j \cong M_t \neq 0$ for all $j \geq 0$.
\end{proof}

Next, we prove the following rigidity theorem.
\begin{theorem}\label{rigidity}
	Let $M= \bigoplus_{n \in \Z} M_n$ be a generalized Eulerian $A_m(A)$-module. If $m \geq 2$, then the following assertions are equivalent:
	\begin{enumerate}[\rm (i)]
		\item $M_n\neq 0$ for all $n \in \Z$.
		\item There exists $r$ with $-m < r<0$ such that $M_r \neq 0$.
	\end{enumerate}
\end{theorem}

\begin{proof}
	$(i) \implies (ii)$ is clear. We only have to prove $(ii) \implies(i)$. We use induction on $m$.
	
	Let $m= 2$. Then we have $M_{-1}\neq 0$.
	
	\begin{claim}
		$ M_i \neq 0$ for infinitely many $i<0$.
	\end{claim} 
	
	Suppose if possible Claim 1 is false. Then by Theorem \ref{rigid-graded} we get $M_i=0$ for all $i \leq -2$. Now we have exact sequence
	\begin{equation}\label{rg1}
	0 \rt H_1(\partial_2; M)_i \rt M_{i+1} \overset{\partial_2}{\lrt} M_i \rt H_0(\partial_2; M)_i\rt 0.
	\end{equation}
	Therefore $H_\nu(\partial_2; M)_i=0$ for all $i \leq -3$ and $\nu=0, 1$. Moreover, by Proposition \ref{pe} we have $H_\nu(\partial_2; M)(-1)$ is a generalized Eulerian $A_{1}(A)$-module. Therefore by Theorem \ref{rigid-graded} it follows that $H_\nu(\partial_2; M)(-1)_i=0$ for all $i \leq-1$ and hence $H_\nu(\partial_2; M)_i=0$ for $i \leq -2$. From the exact sequence \eqref{rg1} for $i=-2$, we get $M_{-1} \cong M_{-2} =0$ which contradicts our hypothesis. So Claim 1 is correct. Thus by Theorem \ref{rigid-graded} we get $M_i \neq 0$ for $i \leq -2$.
	
	\begin{claim}
		$ M_i \neq 0$ for infinitely many $i \geq 0$.
	\end{claim}
	
	Suppose if possible Claim 2 is false. Then by Theorem \ref{rigid-graded} we get $M_i =0$ for all $i \geq 0$. Consider the exact sequence 
	\begin{equation}\label{rg2}
	0 \to H_1(X_2; M)_i \to M_{i-1} \overset{X_2}{\lrt} M_i \to H_0(X_2; M)_i \to 0.
	\end{equation} 
	So $H_1(X_2; M)_i=0= H_0(X_2; M)_i$ for all $i \geq 1$. Moreover, by Proposition \ref{ve} we have $H_\nu(X_2; M)$ is generalized Eulerian for $\nu=0, 1$. Thus by Theorem \ref{rigid-graded} it follows that $H_\nu(X_2; M)_i=0$ for all $i \geq 0$. From exact sequence \eqref{rg2} for $j=0,$ we get \[M_{-1} \cong M_0 =0.\] This contradicts our hypothesis. So Claim 2 is true. 
	
	Thus by Theorem \ref{rigid-graded} it follows that $M_i \neq 0$ for all $i \geq 0$. Hence the result is true when $m=2$. 
	
	We now assume $m \geq 3$ and the result is known for $m-1$. We have $M_{r} \neq 0$ for some $r$ with $-m< r <0$. We want to show $M_i \neq 0$ for all $i \in \Z$.
	
	\begin{claim} 
		$ M_i \neq 0$ for infinitely many $i<0$.
	\end{claim}	
	
	Suppose if possible Claim 3 is false. Then by Theorem \ref{rigid-graded} we get $M_i = 0$ for all $i \leq -m$. Consider the exact sequence 
	\begin{equation}\label{rg3}
	0 \rt H_1(\partial_m; M)_i \rt M_{i+1} \overset{\partial_m}{\lrt} M_i \rt H_0(\partial_m; M)_i\rt 0.
	\end{equation}
	Therefore $H_\nu(\partial_m; M)_i=0$ for all $i \leq -m-1$ and for $\nu=0, 1$. Moreover, by Proposition \ref{pe} we have $H_\nu(\partial_m; M)(-1)$ is a generalized Eulerian $A_{m-1}(A)$-module. Thus for $\nu=0, 1;$ by Theorem \ref{rigid-graded} it follows that $H_\nu(\partial_m; M)(-1)_i=0$ for all $i \leq-m+1$ and by induction hypothesis we get that $H_\nu(\partial_m; M)(-1)_i=0$ for $-m+1< i \leq -1$. Hence $H_\nu(\partial_m; M)_i=0$ for $i \leq -2$. Now by exact sequence \eqref{rg3} we get \[M_{-1} \cong M_{-2} \cong \cdots \cong M_{-m+1} \cong M_{-m}.\] This implies $M_{-m} \cong M_r \neq 0$, a contradiction to our assumption. So Claim 3 is correct. Thus by Theorem \ref{rigid-graded} it follows that $M_i \neq 0$ for $i \leq -m$.
	
	\begin{claim}
		$ M_i \neq 0$ for infinitely many $i \geq 0$.
	\end{claim}	
	
	Suppose if possible Claim 4 is false. Then by Theorem \ref{rigid-graded} it follows that $M_i = 0$ for all $i \geq 0$. Consider the exact sequence 
	\begin{equation}\label{rg4}
	0 \rt H_1(X_m; M)_i \rt M_{i-1} \overset{X_m}{\lrt} M_i \rt H_0(X_m; M)_i\rt 0
	\end{equation}
	So $H_1(X_m; M)_i=0= H_0(X_m; M)_i$ for all $i \geq 1$. Moreover, from Proposition \ref{ve} we have $H_\nu(X_m; M)$ is generalized Eulerian for $\nu=0, 1$. 
By Theorem \ref{rigid-graded} it follows that $H_\nu(X_m; M)_i=0$ for all $i \geq 0$ and by induction hypothesis we get $H_\nu(X_m; M)_i=0$ for $-m+1< i <0$. Thus $H_\nu(X_m; M)_i=0$ for $i \geq -m+2$. Now by exact sequence \eqref{rg4} we get \[M_{-m+1} \cong M_{-m+2} \cong \cdots \cong M_{-1} \cong M_{0}.\] This implies $M_{r} \cong M_0 = 0$, a contradiction to our hypothesis. So Claim 4 is correct. Thus by Theorem \ref{rigid-graded} it follows that $M_i \neq 0$ for $i \geq 0$.
	
	As $m \geq 3$ we also have to prove that if $c \ne r$ and $-m< c < 0$,
then $M_{c} \neq 0$. Suppose if possible $M_c =0$. We will consider the following two cases:
	
	\vspace{0.15cm}
	\noindent
	{\it Case $1$:} $c< r$.
	
	\noindent
	By Proposition \ref{ve} we have $H_1(X_m; M)$ is a generalized Eulerian $A_{m-1}(A)$-module. Moreover, by exact sequence \eqref{rg4} we get $H_1(X_m; M)_{c+1}=0$. Notice $-m+1< c+1 \leq r<0$. 
So by induction hypothesis $H_1(X_m; M)_i=0$ for $-m+1<i<0$. 
	
	Again by exact sequence \eqref{rg4} we also get $H_0(X_m; M)_{c}=0$. Moreover, by Proposition \ref{ve} we have $H_0(X_m; M)$ is a generalized Eulerian $A_{m-1}(A)$-module. We will consider two sub-cases.
	
	\noindent
	{\it Sub-case $1.1$:} $-m+1< c<r<0$.
	
	\noindent
	By induction hypothesis we have $H_0(X_m; M)_i=0$ for $-m+1<i<0$. Thus by exact sequence \eqref{rg4} we get, \[M_{-1} \cong M_{-2} \cong \cdots \cong M_{-m+2} \cong M_{-m+1}.\] Hence $M_r \cong M_c =0$, a contradiction to our hypothesis.
	
	\noindent
	{\it Sub-case $1.2$:} $c=-m+1$.
	
	\noindent
	We have $H_{0}(X_m; M)_{-m+1}=0$. Therefore by induction hypothesis we have \[H_0(X_m; M)_i=0 \quad\text{for} -m+1<i<0.\] So by an argument similar to Sub-case $1.1$, we get $M_r \cong M_c=0$, a contradiction to our hypothesis.
	
	Thus our assumption is false. Hence $M_c \neq 0$ for 
$-m<c<r$.
	
	\vspace{0.15cm}
	\noindent
	{\it Case $2$:} $c>r$.
	
	\noindent
	By Proposition \ref{pe} we have $H_1(\partial_m; M)(-1)$ is a generalized Eulerian $A_{m-1}(A)$-module. Moreover, by exact sequence \eqref{rg3} we get that $H_1(\partial_m; M)(-1)_c= H_1(\partial_m; M)_{c-1}=0$. Notice $-m+1\leq r< c<0$.
So by induction hypothesis $H_1(\partial_m; M)(-1)_i=0$ for $-m+1<i<0$. Thus $H_1(\partial_m; M)_i=0$ for $-m<i<-1$.
	
	Again by exact sequence \eqref{rg3} we also get $H_0(\partial_m; M)_c=0$. Moreover, by Proposition \ref{pe} we have $H_0(\partial_m; M)(-1)$ is a generalized Eulerian $A_{m-1}(A)$-module. 
We consider two sub-cases:
	
	\noindent
	{\it Sub-case $2.1$:} $c \neq -1$.
	
	\noindent
	We have $-m+2<c+1<0$
and $H_0(\partial_m; M)(-1)_{c+1}=0$.
So by induction hypothesis we get $H_0(\partial_m; M)(-1)_i=0$ for $-m+1<i<0$. Thus $H_0(\partial_m; M)_i=0$ for $-m<i<-1$. Therefore by exact sequence \eqref{rg3} we get \[M_{-1} \cong M_{-2} \cong \cdots \cong M_{-m+2} \cong M_{-m+1}.\] Hence $M_r \cong M_c=0$, a contradiction to our hypothesis.
	
	\noindent
	{\it Sub-case $2.2$:} $c = -1$.
	
	\noindent
	Then we have $H_0(\partial_m; M)(-1)_0=0$. 
	So by induction hypothesis $H_0(\partial_m; M)(-1)_i=0$ for $-m+1< i <0$ and hence $H_0(\partial_m; M)_i=0$ for $-m< i <-1$. Therefore by an argument similar to Sub-case $2.1$, we get $M_r \cong M_c=0$, a contradiction to our hypothesis.
	
	Thus our assumption is false. Hence $M_c \neq 0$ for $r<c<0$. The result follows.
\end{proof}

\section{Graded components of local cohomology modules}\label{examples}
This section is devoted in studying graded components of local cohomology modules under the following setup.

\s \label{sa_GE_comm} Let $A$ be a commutative, Noetherian ring containing a field $K$ of characteristic zero. Let $R= A[X_1, \ldots, X_m]$ be standard graded with $R_0=A$ and $m \geq 1$. Let $A_m(A)$ be the $m^{th}$ Weyl algebra over $A$. Then $A_m(A)$ has natural $\Z$-grading given by $\deg z=0$ for all $z \in A$, $\deg X_i=1$, and $\deg \partial_i=-1$. Let $\TT$ be a graded Lyubeznik functor on $^*\Mod(R)$. Set $M:= \TT(R)= \bigoplus_{n \in \Z}M_n$.

The following result from \cite{TP2} is extremely useful to us.

\begin{theorem}\label{GE_gen}\cite[Theorem 3.6]{TP2}
Let $\TT(R)$ be a graded Lyubeznik functor on ${}^*\Mod(R)$. Then $\TT(R)$ is a generalized Eulerian $A_m(A)$-module.
\end{theorem} 

We get the following result as an important outcome 
of Proposition \ref{inj_map}.

\begin{corollary}[with hypothesis as in \ref{sa_GE_comm}]\label{vinj}
Further, assume that $K$ is uncountable. Then there exists some homogeneous element $\eta \in K[X_1, \ldots, X_m]$ of degree $1$ such that $M_n \overset{\cdot\eta}{\rt} M_{n+1}$ is an injective map for all $n \geq -m+1$, and some homogeneous element $\xi \in K[\partial_1, \ldots, \partial_m]$ of degree $-1$ such that $M_n \overset{\cdot\xi}{\rt} M_{n-1}$ is an injective map for all $n \leq -m$.
\end{corollary}

\begin{proof}
From Lemmas \ref{lyucg} and \ref{vpcg}, we get that $M$ is a countably generated $A_m(A)$-module. Moreover, from Theorem \ref{GE_gen} we get that $M$ is a generalized Eulerian $A_m(A)$-module. So the result follows from Proposition \ref{inj_map}.
\end{proof}

\begin{remark}\label{uncountable}
If $K$ is not uncountable, then consider the flat extension $A \to A[[Y]]_{Y}$. Set $C:= A[[Y]]_Y$. From \cite[Remark 3.1, Proposition 3.2]{TP4}, we get that $C$ contains a uncountable field $K[[Y]]_{Y}$, and $C$ is a faithfully flat extension of $A$. We put $S:=C[X_1, \ldots, X_m]$ and $N:=M \otimes_R S$. From Lemma \ref{lyucg} we have $M$ is a countably generated $R$-module. Thus $N$ is a countably generated $S$-module. Furthermore, the flat extension $A \to C$ induces a flat homomorphism $R \to S$. In view of Sub-section \ref{fglyu}, $\TT \otimes_R S$ is a graded Lyubeznik functor on $^*\Mod(S)$. Set $\widehat{\TT}= \TT \otimes_R S$. Notice $N= \widehat{\TT}(S)= M \otimes_R S= M \otimes_A C= \bigoplus_{n \in \Z} (M_n \otimes_A C)$. As $C$ is a faithfully flat extension of $A$ so we have $M=0$ if and only if $N=0$, and $M_n=0$ if and only if $M_n\otimes_A C=0$.
\end{remark}

Vanishing of the components 
has a significant effect on the graded local cohomology module. 
\begin{theorem}[with hypothesis as in \ref{sa_GE_comm}]\label{vanishing}
If $M_n=0$ for all $|n|\gg0$, then $M=0$.
\end{theorem}

\begin{proof} 
In view of Remark \ref{uncountable}, it is enough to prove the result considering $K$ is uncountable. From Theorem \ref{GE_gen}, we have $M$ is a generalized Eulerian $A_m(A)$-module. If $M \neq 0$, then by Corollary \ref{cgvanish} it follows that $M_n \neq 0$ for infinitely many  $n\gg 0$ OR $M_n \neq 0$ for infinitely many $n \ll 0$, a contradiction.
\end{proof}

The following theorem 
indicates that $M=\TT(R)$ is tame.
\begin{theorem}[with hypothesis as in \ref{sa_GE_comm}]\label{tameness}
Then
\begin{align*}
&M_{n_0} \neq 0 \text{ for some } n_0 \geq -m+1 \implies M_{n} \neq 0 \text{ for all } n \geq {n_0},\\
&M_{n_0} \neq 0 \text{ for some } n_0 \leq -m \implies M_{n} \neq 0 \text{ for all } n \leq {n_0}.
\end{align*}
\end{theorem}
\begin{proof} 
In view of Remark \ref{uncountable}, it is enough to prove the result considering $K$ is uncountable. From Theorem \ref{GE_gen}, we have $M$ is a generalized Eulerian $A_m(A)$-module. By Corollary \ref{cgtame}, we get the required result.
\end{proof}

An important consequence of the Theorem \ref{rigid-graded} and Theorem \ref{rigidity} is shown bellow. 

\begin{theorem}[Rigidity]
	{\rm (}with hypothesis as in \ref{sa_GE_comm}{\rm)}. \label{rigid-LC_graded} Then 
	\begin{enumerate}[\rm(I)]
		\item \Tfae
		\begin{enumerate}[\rm(a)]
			\item $M_n \neq 0$ for infinitely many $n<0$.
			\item There exists $r$ such that $M_n \neq 0$ for all $n \leq r$.
			\item $M_n \neq 0$ for all $n \leq -m$.
			\item $M_t \neq 0$ for some $t \leq -m$.
		\end{enumerate}
		\item \Tfae
		\begin{enumerate}[\rm(a)]
			\item $M_n \neq 0$ for infinitely many $n \geq 0$.
			\item There exists  $s$ such that $M_n \neq 0$ for all $n \geq s$.
			\item $M_n \neq 0$ for all $n \geq 0$.
			\item $M_t \neq 0$ for some $t \geq 0$.
		\end{enumerate}
		\item If $m \geq 2$, then the following conditions are equivalent:
		\begin{enumerate}[\rm(a)]
			\item $M_n\neq 0$ for all $n \in \Z$.
			\item There exists $t$ with $-m < t<0$ such that $M_t \neq 0$.
		\end{enumerate}
	\end{enumerate}
\end{theorem}

\begin{proof}
From Theorem \ref{GE_gen}, we have $M$ is a graded generalized Eulerian $A_m(A)$-module. Therefore $M$ satisfies properties $\rm(I)$ and $\rm(II)$ by Theorem \ref{rigid-graded}, and property (III) by Theorem \ref{rigidity}. 
\end{proof}

\begin{remark}
The above result states that we can get complete information about the nonzero graded components of $M=\mathcal{T}(R)$, by checking only the three components $M_0, M_{-1}$, and $M_{-m}$.
\end{remark}

%

As an instant 
sequel of Theorem \ref{rigid-LC_graded}, we get the asymptotic stability of dimensions of supports. Let $T=\bigoplus_{n \geq 0} T_n$ be a standard graded Noetherian ring, $T_+:=\bigoplus_{n>0}T_n$ denote the irrelevant ideal of $T$, and let $L=\bigoplus_{n \in \Z} L_n$ be a finitely generated graded $T$-module. In \cite{Brod1}, Brodmann considered the question 
regarding the existence of an integer $n_0$ such that $\dim_{T_0} H^i_{T_+}(L)_n =\dim_{T_0} H^i_{T_+}(L)_{n_0}$ for all $n \leq n_0$. Although he 
gave positive answers for a few cases especially when $\dim T_0 \leq 2$, this problem is still open in general. 

\begin{theorem}[with standard assumption \ref{sa_GE_comm}]\label{stable_supp}
	Then 
	\begin{enumerate}[\rm (i)]
		\item $\dim_A M_n= \dim_A M_{-m}$ for all $n \leq -m$.
		\item $\dim_A M_n= \dim_A M_0$ for all $n \geq 0$.
		\item If $m \geq 2$ and $-m <r, s<0$, then 
		\begin{enumerate}[\rm (a)]
			\item $\dim_A M_r= \dim_A M_{s}$.
			\item $\dim_A M_r\leq \min\{ \dim_A M_{-m}, \dim_A M_0\}$.
		\end{enumerate}
	\end{enumerate}
\end{theorem}

\begin{proof}
Let $P$ be a prime ideal of $A$. In view of Subsection \ref{fglyu}, $\mathcal{T}(-)_P$ is a graded Lyubeznik functor on ${}^*\Mod\left(A_P[X_1, \ldots, X_m]\right)$.
	
(i) Take $n \leq -m$. By Theorem \ref{rigid-LC_graded}(I) we get that $\left(M_{-m}\right)_P=\left(M_P\right)_{-m}>0$ if and only if $\left(M_n\right)_P=\left(M_P\right)_n>0$. The result follows.
	
(ii) and (iii) follow from Theorem \ref{rigid-LC_graded}(II) and (III) respectively.
\end{proof}


\section{Examples}

Let $A$ be a Noetherian ring containing a field of characteristic zero. Let $R= A[X_1, \ldots, X_m]$ be standard graded with $R_0=A$ and $m \geq 1$. 
Let $I$ be a homogeneous ideal in $R$. Set $M:= H_I^i(R)= \bigoplus_{n \in \Z}M_n$. In this section, we produce a few examples to show that some results 
due to Puthenpurakal \cite{TP2} are false if $A$ is not regular.

\s In \cite[Theorem 1.7]{TP2}, it is shown that if $A$ is a regular domain, $I \cap A \neq 0$, and $M_n \neq 0$, then $M_n$ is NOT finitely generated as an $A$-module. This is not true in general.
\begin{example*}
Let $(A, \m)$ be a local domain with dimension $d$ such that $H_{\m}^i(A)$ is finitely generated and non-zero for some $i<d$. Take $I=\m R$. Then $H_I^i(R)_0=H_{\m}^i(A)$ is non-zero and finitely generated as an $A$-module. 
\end{example*}

\s In \cite[Theorem 1.8]{TP2} it is shown that if $A$ is a regular ring, then for every prime ideal $P$ of $A$,
either $\mu_j(P, M_n)=\infty$ for all $n \in \Z$ or $\mu_j(P, M_n)<\infty$ for all $n \in \Z$. This does not hold in general.
\begin{example*}\label{eg2}
There exists example of Noetherian local ring $(A, \m)$ with $H_J^i(A)$ having infinite zeroth Bass number over $\m$ for some ideal $J$ in $A$, see \cite{RH} (also see \cite[Chapter 20, Section 5]{TFLC}). Now $M=H_{JR}^i(R)=H_J^i(A) \otimes_A R= \bigoplus_{n \in \N} M_n$. So in this case $\mu_0(\m, M_n)=0$ for $n<0$ (as $M_n=0$) and $\mu_0(\m, M_n)$ is infinite for $n \geq 0$.	
\end{example*}

\s Assume $m=1$. In \cite[Theorem 1.9]{TP2} it is shown that if $A$ is a regular ring and $P$ is a prime ideal,
then $\mu_j(P, M_n)<\infty$ for all $n \in \Z$. This is not true if $A$ is not regular.
\begin{example*}
In Example \ref{eg2}, take $m=1$.
\end{example*}

\s In \cite[Theorem 1.13]{TP2} it is shown that if $A$ is regular local or a smooth affine $K$-algebra then $\bigcup_{n \in \Z}\Ass_A M_n$ is a finite set. This is false in general.
\begin{example*}
There exists affine or local ring $A$ such that $\Ass_A(H_I^2(A))$ is an infinite set. For example, take $K$ be an arbitrary field and consider the hypersurface \[A=K[s, t, u, v, x, y]/(sv^2x^2-(s+t)vxuy+tu^2y^2).\] Then $H^2_{(x, y)}(A)$ has infinitely many associated primes, see \cite[Example 22.17]{TFLC}. Localization of $A$ at the homogeneous maximal ideal $(s, t, u, v, x, y)$ gives a local example. Take $I=(x, y)R$. Since $M=H_{I}^2(R)=H_{(x, y)}^2(A) \otimes_A R= \bigoplus_{n \in \N} M_n$ so in both cases $\Ass_A M_0$ is infinite.
\end{example*}

\s In \cite[Theorem 1.14]{TP2} it is shown that if $A$ is regular then 
\begin{equation*}\label{inj-dim-rel}
\injdim M_n \leq \dim M_n \quad \text{ for all } n \in \Z.
\end{equation*} 
It is not true in general, as the following example shows.
\begin{example*}
Let $(A, \m)$ be a Noetherian local ring containing a field with $\dim A >0$ and $\depth A =0$. Note that $M=H_{\m R}^0(R)=H_{\m}^0(A) \otimes_A R$. Thus $M_0=H_{\m}^0(A)$ and hence $M_0$ is a finitely generated $A$-module. If $\injdim M_0< \infty$ then by \cite[Corollary 9.6.2]{BJ} we get that $A$ is Cohen-Macaulay, a contradiction. So $\injdim M_0$ is infinite.
\end{example*}


\section{Graded local cohomology module with respect to a pair of ideals}
In this section, we turn our attention to the local cohomology modules with respect
to a pair of ideals and show that they are also generalized Eulerian. First, we recollect basic concepts and some preliminary results that we need. 

\s \label{recall-Lc-pair}{\it Recall}.
\begin{enumerate}[\rm \bf I.]
	
\item\label{itm:1} Let $B$ be a commutative Noetherian ring, 
$I, J$ ideals in $B$, and $N$ a $B$-module. Set 
\[\Gamma_{I,J}(N)=
\{y \in N \mid I^ny \subseteq J y \mbox{ for some } n \geq 1\}.\] 
One can easily check that $\Gamma_{I,J}(N)$ is a submodule of $N$. We call $\Gamma_{I,J}(-)$ the {\it $(I, J)$-torsion functor}. 
Notice that $\Gamma_{I,J}(-)=\Gamma_{I}(-)$ if $J=0$. By \cite[Lemma 1.2]{TYY}, 
$\Gamma_{I,J}(-)$ is a left exact functor. For $i \geq 0$, the $i$-th right derived functor of $\Gamma_{I,J}(-)$ is 
called {\it the $i$-th local cohomology functor} and is denoted by $H^i_{I,J}(-)$. We refer to $H^i_{I,J}(N)$ 
as the {\it $i$-th local cohomology module of $N$ with respect to $(I,J)$}.

Consider 
$W(I, J)=\{\p \in \Spec (B) \mid I^n \subseteq \p +J \mbox{ for some integer } n \geq 1\}$.
Generally, $W(I, J)$ is not  a closed subset of $\Spec(B)$. From \cite[Corollary 1.8]{TYY} we have $y \in \Gamma_{I,J}(N)$ if and only if $\Supp (By) \subseteq W(I,J)$.

\item \label{itm: I} In \cite{TYY}, 
Takahashi et al 
generalized the \v Cech complex to obtain the local cohomology modules with respect a pair of ideals as the homologies of the complex. For an element $b \in B$, consider the subset $S_{b, J}=\{b^n+j\mid n \geq 0, j \in J\}$. It can be checked that $S_{b, J}$ is a multiplicatively closed subset of $B$. By $N_{b, J}:=S^{-1}_{b, J} N$ we denote the module of fractions of $N$ with respect to $S_{b, J}$. 
For any sequence of elements ${\bf b}:=b_1,\ldots, b_s$ of $B$, define a complex
\begin{equation}\label{gen-cech}
{\bf C}^{\bullet}_{{\bf b}, J}: ~ 0 \to B \to \prod_{i=1}^s B_{b_i, J} \to \prod_{i<j} \big(B_{b_i, J}\big)_{b_j, J} \to \cdots \to \big( \cdots \big(B_{b_, J}\big) \cdots \big)_{b_s, J} \to 0
\end{equation}
by setting ${\bf C}^{\bullet}_{{\bf b}, J}=\otimes_{i=1}^s {\bf C}^{\bullet}_{b_i, J}$, where ${\bf C}^{\bullet}_{b_i, J}=\big( 0 \to B \to B_{b_i, J} \to 0\big)$. If $I=(b_1, \ldots, b_s)$, then from \cite[Theorem 2.4]{TYY} we have 
\[H^i_{I,J}(N) \cong H^i\big({\bf C}^{\bullet}_{{\bf b}, J} \otimes_B N\big).\]

\item Let $T$ be a graded ring, $I$ be a homogeneous ideal, and $J$ be an arbitrary ideal in $T$. Let $L$ be a graded $T$-module. In \cite{LimPer}, Lima and Jorge P\'erez introduced a grading on $\Gamma_{I,J}(L)$ 
by setting
\begin{equation}\label{lim-grading}
\Gamma_{I,J}(L)_u=\{y \in L_u \mid I^n y \subseteq J y \mbox{ for some } n \geq 1\} \quad \mbox{ for all } u \in \Z.
\end{equation}
Define ${}^*H^i_{I,J}(L)$ to be the $i$-th right derived functor of $\Gamma_{I,J}(L)$ on ${}^*\Mod(T)$. Then they showed that for all $i$,
\begin{equation}\label{lim-rel}
{}^*H^i_{I,J}(L) \cong H^i_{I,J}(L)
\end{equation}
as underlying $T$-modules, see \cite[Proposition 2.5]{LimPer}. 

\end{enumerate}


\s We now define a graded version of the generalized \v Cech complex (see \ref{recall-Lc-pair} \ref{itm: I}) under the following assumptions.

\noindent
{\bf S1.} Let $T$ be a standard graded Noetherian ring and $I, J$ both are homogeneous ideals in $T$. Let $L$ be a graded $T$-module. 

For any homogeneous element $b$ of $T$, set \[{}^*S_{b, J}=\{b^n+j \mid n \geq 0, j \mbox{ is a homogeneous element of }J \mbox{ with } \deg j=\deg b^n\}.\]
We put ${}^*L_{b, J}={}^*S^{-1}_{b, J} L$. Let $n_1, n_2 \geq 0$ and $j_1, j_2 \in J$ be homogeneous elements. We have \[(b^{n_1}+j_1) \cdot (b^{n_2}+ j_2)=b^{n_1+n_2} +(j_2b^{n_1}+j_1b^{n_2}+j_1j_2)=b^{n_1+n_2}+j\] 
with $j \in J$. Clearly, $\deg j =\deg b^{n_1+n_2}$. Hence ${}^*S_{b, J}$ is a multiplicativly closed subset of $T$. We can define a complex ${}^*{\bf C}^{\bullet}_{{\bf b}, J}$ similarly as 
\eqref{gen-cech} replacing $S_{b, J}$ by ${}^*S_{b, J}$.

The following theorem is crucial to prove our main result of this section.
\begin{theorem}[with hypothesis as in S1]\label{main}
There is a graded isomorphism
\[H^i_{I,J}(L) \cong H^i\big({}^*{\bf C}^{\bullet}_{{\bf b}, J} \otimes_T L\big).\]
\end{theorem}

Using a graded version of the arguments 
given in the proof of \cite[Theorem 2.4]{TYY}, 
the above result follows. However, for the sake of completeness, we give a proof here.


First, we state 
the graded version of a few results of the paper \cite{TYY} that 
we 
require to prove Theorem \ref{main}. We set
\[{}^*W(I,J)=\{\p \in W(I,J) \mid \p \mbox{ is homogeneous}\}.\]
Following the same lines of proof 
of \cite[Proposition 1.11]{TYY}, one can show the following.

\begin{lemma}[with hypothesis as in S1]\label{injhull-torsion}
For any homogeneous prime ideal $\p$ of $T$, ${}^*E(T/\p)$ is an $(I,J)$-torsion $T$-module if $\p \in {}^*W(I,J)$, and ${}^*E(T/\p)$ is an $(I,J)$-torsion free $T$-module if $\p \notin {}^*W(I,J)$.
\end{lemma}
We further need the following lemmas.

\begin{lemma}[with hypothesis as in S1]\label{rel-W-S}
A homogeneous prime ideal $\p \in {}^*W(I,J)$ if and only if $\p \cap {}^*S_{b,J} \neq \phi$ for any homogeneous element $b \in I$.
\end{lemma}

\begin{proof}
Let $\p \in {}^*W(I,J)$. Then $I^n \subseteq \p +J$ for some $n \geq 1$.	So for any homogeneous element $b \in I$, we have $b^n \subseteq \p+J$, that is, there are some $c \in J$ and $d \in \p$ such that $b^n=c+d$. Since $J$ and $\p$ are homogeneous ideals, we can choose $c, d$ to be homogeneous. Clearly, $\deg c=\deg b^n$ and hence $d= b^n-c \in {}^*S_{b,J} \cap \p$.
	
For the converse, take any homogeneous element $b$ of $I$. Let $d \in {}^*S_{b,J} \cap \p$. Then $d=b^n+c$ for some $n \geq 1$ and some homogeneous element $c \in J$ with $\deg b^n=\deg c$. Thus $b^n=d - c \in \p +J$. As this holds for any homogeneous element $b$ of $I$ 
and $I$ is a finitely generated homogeneous ideal so $I^{n'} \subseteq \p +J$ for some large $n'$, that is, $\p \in {}^*W(I,J)$.
\end{proof}

\begin{lemma}[with hypothesis as in S1]\label{long-cech}
Let $0 \to L \to M \to N \to 0$ be a short exact sequence of graded $T$-modules. Let ${\bf b}:= b_1, \ldots, b_s$ be a sequence of homogeneous elements in $T$. Then there exists a natural long exact sequence of graded $T$-modules
\[\cdots \to H^i\big({}^*{\bf C}^{\bullet}_{{\bf b}, J} \otimes_T L\big) \to H^i\big({}^*{\bf C}^{\bullet}_{{\bf b}, J} \otimes_T M\big) \to H^i\big({}^*{\bf C}^{\bullet}_{{\bf b}, J} \otimes_T N\big) \to H^{i+1}\big({}^*{\bf C}^{\bullet}_{{\bf b}, J} \otimes_C L\big) \to \cdots.\]
\end{lemma}
\begin{proof}
Consider the following commutative diagram
\[\xymatrix@C=1.2em@R=1em{0 \ar[r]& L \ar[r]\ar[d]& \prod_{i=1}^s {}^*L_{b_i, J} \ar[r]\ar[d]&  \prod_{i<j} \big({}^*L_{b_i, J}\big)_{b_j, J} \ar[r]\ar[d]&  \cdots \ar[r]&  \big( \cdots \big({}^*L_{b_, J}\big) \cdots \big)_{b_s, J} \ar[r]\ar[d]& 0\\
0 \ar[r]& M \ar[r]\ar[d]& \prod_{i=1}^s {}^*M_{b_i, J} \ar[r]\ar[d]&  \prod_{i<j} \big({}^*M_{b_i, J}\big)_{b_j, J} \ar[r]\ar[d]&  \cdots \ar[r]&  \big( \cdots \big({}^*M_{b_, J}\big) \cdots \big)_{b_s, J} \ar[r]\ar[d]& 0\\
0 \ar[r]& N \ar[r]& \prod_{i=1}^s {}^*N_{b_i, J} \ar[r]&  \prod_{i<j} \big({}^*N_{b_i, J}\big)_{b_j, J} \ar[r]&  \cdots \ar[r]&  \big( \cdots \big({}^*N_{b_, J}\big) \cdots \big)_{b_s, J} \ar[r]& 0}\]
with exact rows and columns. From this we get an exact sequence of cochain complexes
\[0 \to {}^*{\bf C}^{\bullet}_{{\bf b}, J} \otimes_T L \to {}^*{\bf C}^{\bullet}_{{\bf b}, J} \otimes_T M \to {}^*{\bf C}^{\bullet}_{{\bf b}, J} \otimes_T N \to 0\]
which induces the expected long exact sequence.
\end{proof}

\begin{lemma}\label{long-LC}\cite[Section 2, (D)]{LimPer}
Let $T$ be a graded ring, $I$ be a homogeneous ideal, and $J$ be an arbitrary ideal in $T$. For any short exact sequence $0 \to L \to M \to N \to 0$ of graded $T$-modules, there exists a natural long exact sequence 
\[\cdots \to {}^*H^i_{I,J}(L) \to {}^*H^i_{I,J}(M) \to {}^*H^i_{I,J}(N) \to {}^*H^{i+1}_{I,J}(L) \to \cdots\]
of graded $T$-modules.
\end{lemma}

\begin{lemma}[with hypothesis as in S1]\label{injhull-cohom-zero}
For every ${}^*$injective $T$-modules $I$,
\[H^i({}^*{\bf C}^{\bullet}_{{\bf b}, J} \otimes_T I)=0 \quad \mbox{ for each } i > 0.\]
\end{lemma}

\begin{proof}
It is enough to result the claim for ${}^*E:={}^*E_T(T/\p)$, where $\p$ is a homogeneous prime ideal of $T$. Suppose $s=1$. Consider the complex,
\[{}^*{\bf C}^{\bullet}_{b_1, J} \otimes_T {}^*E: ~ 0 \to {}^*E \to {}^*E_{b_1, J} \to 0.\]
In view of Lemma \ref{rel-W-S}, ${}^*E_{b_1, J}=\{0\}$ if $\p \in W((b_1), J)$, and ${}^*E_{b_1, J} \cong {}^*E$ otherwise. 

Next, suppose $s>1$. We put ${\bf b}':=b_2, \ldots, b_s$. Notice that ${}^*{\bf C}^{\bullet}_{{\bf b}, J}= {}^*{\bf C}^{\bullet}_{b_1, J} \otimes_T {}^*{\bf C}^{\bullet}_{{\bf b}', J}$. Consider the first quadrant double complex

\[\xymatrix@C=1.5em@R=1em{ & \vdots \ar[d] & \vdots \ar[d] & \\
0 \ar[r] &  T \otimes_T \left({}^*{\bf C}^{q+1}_{{\bf b}', J} \otimes_T {}^*E\right) \ar[r] \ar[d] & {}^*T_{b_1, J} \otimes_T \left({}^*{\bf C}^{q+1}_{{\bf b}', J} \otimes_T {}^*E\right) \ar[r] \ar[d] & 0\\
0 \ar[r] & T \otimes_T \left({}^*{\bf C}^{q}_{{\bf b}', J} \otimes_T {}^*E\right) \ar[r] \ar[d] & {}^*T_{b_1, J} \otimes_T \left({}^*{\bf C}^{q}_{{\bf b}', J} \otimes_T {}^*E\right) \ar[r] \ar[d] & 0\\
0 \ar[r] & T \otimes_T \left({}^*{\bf C}^{q-1}_{{\bf b}', J} \otimes_T {}^*E\right) \ar[r] \ar[d] & {}^*T_{b_1, J} \otimes_T \left({}^*{\bf C}^{q-1}_{{\bf b}', J} \otimes_T {}^*E\right) \ar[r] \ar[d] & 0\\
	& \vdots & \vdots &}
\]
Since both $T$ and $T_{b_1, J}$ are free over $T$, 
\[E^1_{p,q}=H^q\left({}^*{\bf C}^{p}_{b_1, J}\otimes_T \left({}^*{\bf C}^{\bullet}_{{\bf b}', J} \otimes_T {}^*E\right)\right) \cong {}^*{\bf C}^{p}_{b_1, J}\otimes_T H^q\left({}^*{\bf C}^{\bullet}_{{\bf b}', J} \otimes_T {}^*E\right)\]
and hence
\[E^2_{p,q}=H^p\left({}^*{\bf C}^{\bullet}_{b_1, J}\otimes_T H^q\left({}^*{\bf C}^{\bullet}_{{\bf b}', J} \otimes_T {}^*E\right)\right).\]
Clearly, 
\[E^2_{p,q} \implies H^{p+q}\left({}^*{\bf C}^{\bullet}_{{\bf b}, J} \otimes_T {}^*E\right),\]
as this spectral sequence is bounded. By induction hypothesis, $H^q\left({}^*{\bf C}^{\bullet}_{{\bf b}', J} \otimes_T {}^*E\right)=0$ for all $q>0$. Therefore, the spectral sequence collapses at $E^2$. Thus, by \cite[Definition 5.2.7]{W},
\begin{align*}
H^{n}\left({}^*{\bf C}^{\bullet}_{{\bf b}, J} \otimes_T {}^*E\right)
&=H^n\left({}^*{\bf C}^{\bullet}_{b_1, J}\otimes_T H^0\left({}^*{\bf C}^{\bullet}_{{\bf b}', J} \otimes_T {}^*E\right)\right)\\
&=H^n\left({}^*{\bf C}^{\bullet}_{b_1, J}\otimes_T \Gamma_{({\bf b}'), J}( {}^*E)\right)\\
&=H^n\left(0 \to \Gamma_{({\bf b}'), J}( {}^*E) \to \Gamma_{({\bf b}'), J}( {}^*E)_{(a_1, J)} \to 0 \right).
\end{align*}
Hence, $H^{n}\left({}^*{\bf C}^{\bullet}_{{\bf b}, J} \otimes_T {}^*E\right)=0$ for all $n \geq 2$. By Lemma \ref{injhull-torsion}, $\Gamma_{({\bf b}'), J}( {}^*E)={}^*E$ or $\{0\}$. The result follows from the case $s=1$.
\end{proof}

\begin{proof}[Proof of Theorem \ref{main}]
Let $I=(b_1, \ldots, b_s)$, where $b_i$'s are homogeneous elements of $T$. 
\begin{claim*}
$H^0\big({}^*{\bf C}^{\bullet}_{{\bf b}, J} \otimes_T L\big) \cong \Gamma_{I,J}(L)$ as graded $R$-modules. 
\end{claim*}
To prove the claim, we show that $y \in \Gamma_{I,J}(L)_u$ if and only if $ y \in \mathcal{K}_u $ for all $u \in \Z$, where $\mathcal{K}:=\ker \big( L \to \prod_{i=1}^s L_{b_i, J}\big)$. 
Fix $u$. Let $y \in \Gamma_{I,J}(L)_u$.
Then by \eqref{lim-grading}, there exists some integer $n \geq 1$ such that $I^n y \subseteq J y$, i.e., $b_i^n y \in J y$ for $i=1, \ldots, s$. Let $b_i^n y=c_i y$ for some $c_i \in J$. Since $b_i$ and $y$ are homogeneous elements, we can choose $c_i$ to be homogeneous (here, we use the fact that $J$ is a homogeneous ideal).
Clearly, $\deg b_i^n=\deg c_i$ and hence $b_i^n-c_i \in {}^*S_{b_i, J}$. As $(b_i^n-c_i) y=0$ for $i=1, \ldots, s$ so we get $y \in \mathcal{K}$. Conversely, let $y \in \mathcal{K}_u$. Then there exists $n_i \geq 1$ and a homogeneous element $c_i \in J$ such that $\deg b_i^{n_i}=\deg c_i$ and $(b_i^{n_i}-c_i) y=0$. Thus $b_i^{n_i}y=c_i y \in J y$ for $i=1, \ldots, s$ and hence $I^n y \subseteq J y$ for some large $n$, that is, $y \in \Gamma_{I,J}(L)$.


Now, we use induction on $i$. The case $i=0$ follows from the first claim. We now assume that $i>0$. Then, by Lemma \ref{long-cech} and Lemma \ref{long-LC}, the graded short exact sequence
\[0 \to L \to {}^*E_T(L) \to \mathcal{C} \to 0\]
gives us the following commutative diagram
\[\scalebox{0.9}{\xymatrix@C=0.7em@R=1em{
\cdots \ar[r] & H^{i-1}\big({}^*{\bf C}^{\bullet}_{{\bf b}, J} \otimes_T {}^*E_T(L)\big) \ar[r] \ar[d] & H^{i-1}\big({}^*{\bf C}^{\bullet}_{{\bf b}, J} \otimes_T \mathcal{C}\big) \ar[r] \ar[d]& H^{i}\big({}^*{\bf C}^{\bullet}_{{\bf b}, J} \otimes_T L\big) \ar[r] \ar[d] & H^{i}\big({}^*{\bf C}^{\bullet}_{{\bf b}, J} \otimes_T {}^*E_T(L)\big) \ar[r] \ar[d] & \cdots\\
\cdots \ar[r] & {}^*H^{i-1}_{I,J}({}^*E_T(L)) \ar[r] & {}^*H^{i-1}_{I,J}(\mathcal{C}) \ar[r] & {}^*H^i_{I,J}(L) \ar[r] & {}^*H^{i}_{I,J}({}^*E_T(L)) \ar[r] & \cdots}}\]
Combining Lemma \ref{injhull-cohom-zero} and \cite[Proposition 2.4]{LimPer}, we get that $H^{i}\big({}^*{\bf C}^{\bullet}_{{\bf b}, J} \otimes_T {}^*E_T(L)\big)=0={}^*H^{i}_{I,J}({}^*E_T(L))$. Moreover, by induction hypothesis
\[H^{i-1}\big({}^*{\bf C}^{\bullet}_{{\bf b}, J} \otimes_T \mathcal{C}\big) \cong {}^*H^{i-1}_{I,J}(\mathcal{C}), \quad \mbox{and} \quad H^{i-1}\big({}^*{\bf C}^{\bullet}_{{\bf b}, J} \otimes_T {}^*E_T(L)\big) \cong {}^*H^{i-1}_{I,J}({}^*E_T(L)).\]
Hence, by Five lemma we get that $H^{i}\big({}^*{\bf C}^{\bullet}_{{\bf b}, J} \otimes_T L\big) \cong {}^*H^{i}_{I,J}(L)$. In light of \eqref{lim-rel}, the result follows.
\end{proof}

Now we state the main result of this section.
\begin{theorem}
Let $A$ be a commutative Noetherian ring containing a field of characteristic zero. Let $R=A[X_1, \ldots, X_m]$ and $A_m(A)$ be the $m^{th}$ Weyl algebra over $A$. Suppose both $R$ and $A_m(A)$ are standard graded. If $I, J \subseteq R$ are homogeneous ideals, and 
$L$ is a generalized Eulerian $A_m(A)$-module, then $H^i_{I,J}(L)$ is a generalized Eulerian $A_m(A)$-module.
\end{theorem}

\begin{proof}
If $L$ is 
a generalized Eulerian $A_m(A)$-module, then so is ${}^*L_{a, J}$ by 
\cite[Corollary 3.7]{TPJS}. 
By Theorem \ref{main}, it follows that $H^i_{I,J}(L)$ is a generalized Eulerian $A_m(A)$-module for each $i \geq 0$. 
\end{proof}

\begin{remark}
Due to Theorem \ref{rigid-graded} and Theorem \ref{rigidity}, it follows that $H^i_{I,J}(L)$ has rigidity property. Let us remark that the rigidity property of a module implies its vanishing and tameness properties. As a consequence, $H^i_{I,J}(L)$ has vanishing and tameness properties. However, we fail to prove the countable generation of $H^i_{I,J}(L)$ to use Theorem \ref{cgvanish} and Theorem \ref{cgtame} directly. 
\end{remark}


	
\end{document}